\documentclass[a4paper,11pt]{article}

\usepackage{hyperref}
\usepackage[latin1]{inputenc}
\usepackage{latexsym}
\usepackage{amsmath,amssymb,amsfonts,amscd,theorem}
\usepackage{enumerate}
\usepackage{graphicx}
\usepackage{epsfig}

\usepackage{color}
\newtheorem{thm}{Theorem}[section]
\newtheorem{prop}[thm]{Proposition}
\newtheorem{lemma}[thm]{Lemma}
\newtheorem{cor}[thm]{Corollary}
\newenvironment{proof}[1][Proof]%
{
\begin{trivlist} \item[]  {\em #1.} }%
{\hspace*{\fill} $\Box$
\end{trivlist}}

      \newtheorem{definition}[thm]{Definition}
      \newtheorem{remark}[thm]{Remark}

      \newtheorem{claim}[thm]{Claim}
\newcommand{\Id}{\mathrm{Id} }
\newcommand{\1}{\boldsymbol{1} }
\newcommand{\cupd}{\overset{.}{\cup}}
\newcommand{\Ad}{\operatorname{Ad}}
\newcommand{\ad}{\operatorname{ad}}
\newcommand{\image}{\operatorname{Im}}
\newcommand{\rank}{\operatorname{rk}}
\newcommand{\tr}{\operatorname{tr}}
\newcommand{\Ker}{\operatorname{Ker}}

\renewcommand{\hom}{\operatorname{Hom}}

\newcommand{\aut}{\operatorname{Aut}}
\newcommand{\Aut}[1]{\aut(#1)}

\newcommand{\CC}{\mathbb{C}}

\newcommand{\NN}{\mathbb{N}}

\newcommand{\RR}{\mathbb{R}}
\newcommand{\ZZ}{\mathbb{Z}}

\def\co{\colon}
\begin{document}

\title{The infinitesimal projective rigidity under Dehn fillings}

\author{Michael Heusener and 
Joan Porti\footnote{Partially supported
by the Spanish Micinn through grant MTM2009-07594. Prize ICREA 2008}}


\maketitle

\begin{abstract}
To a hyperbolic manifold one can associate a canonical projective structure and a fundamental question is whether or not
it can be deformed. In particular, the canonical projective structure of a finite volume hyperbolic manifold with cusps might have deformations which are trivial on the cusps. 

The aim  of this article is to prove that if the canonical projective structure on a cusped hyperbolic manifold $M$ is \emph{infinitesimally projectively rigid relative to the cusps}, then infinitely many hyperbolic Dehn fillings on $M$ are locally projectively rigid.
We analyze in more
detail the figure eight knot and the Whitehead link exteriors, for which we can 
give explicit infinite families of slopes with projectively rigid Dehn fillings.
\medskip
\noindent \emph{MSC:} 57M50; 53A20; 53C15\\
\emph{Keywords:} Projective structures; variety of representations; infinitesimal deformations.
\end{abstract}


\section{Introduction}
\label{sec:intro}

A closed hyperbolic $n$--dimensional manifold inherits a canonical projective structure. This can be easily seen by 
considering the Klein model for the hyperbolic space. Projective structures on manifolds were studied by 
Benz\'ecri in the 1960's \cite{Benzecri}. 
Though the hyperbolic structure is rigid for $n>2$
(cf.~\cite{Weil,Mostow}), it might be possible to deform 
the canonical projective structure.
Kac and Vinberg \cite{Kac-Vinberg} gave the first examples of 
such deformations.
Koszul \cite{Koszul} and Goldman later generalized these examples. 
Johnson and Millson provide deformations of the canonical projective structure by means of bending along 
totally geodesic surfaces \cite{JohnsonMillson}. 
Examples of deformations for Coxeter orbifolds have been obtained
by Benoist \cite{BenoistIV}, Choi \cite{Choi}, and Marquis \cite{Marquis}. 
See the survey by Benoist \cite{BenoistSurvey} and references therein for more results on
convex projective structures.

In the sequel we will use the following notation:
\begin{definition}
A closed hyperbolic manifold is called \emph{locally projectively rigid} if the canonical projective structure induced by the hyperbolic metric 
cannot be deformed. 
\end{definition}

Cooper, Long and Thistlethwaite have studied the deformability of 4500 hyperbolic manifolds from the
Hodgson--Weeks census with rank 2 fundamental group \cite{CLTEM}, proving that at most 61 can be deformed. 
The goal of this paper is to provide infinite families of  projectively locally rigid manifolds, by means of Dehn filling.

\medskip

Let  $N$ be a closed hyperbolic $3$--dimensional manifold.
We will make use of the fact that geometric structures on $N$ are controlled by their holonomy representation.
Hence we consider the holonomy representation of the closed hyperbolic 3--manifold $N$
$$
\rho \co\pi_1(N)\to PSO(3,1) \subset PGL(4).
$$
If not specified, the coefficients of matrix groups are real: 
$PGL(4)= PGL(4,\mathbf R)$.
The closed manifold $N$ is locally projectively rigid if and only if
all deformations of $\rho$ in $PGL(4)$
are contained
in the $PGL(4)$--orbit  of $\rho$.

Existence or not of deformations is often studied at the infinitesimal level. 
We may consider the adjoint action on the lie algebra $\mathfrak{so}(3,1)$. Then Weil's infinitesimal rigidity \cite{Weil} asserts that
$$
H^1(\pi_1(N);\mathfrak{so}(3,1)_{\Ad\rho})=0.
$$
The adjoint action extends to the Lie algebra 
$\mathfrak{sl}(4) := \mathfrak{sl}(4,\mathbf R)$ and motivates the following definition.

\begin{definition}
A closed hyperbolic three manifold $N$ is called \emph{infinitesimally projectively rigid} if 
$$
H^1(\pi_1(N);\mathfrak{sl}(4)_{\Ad\rho})=0.
$$
\end{definition}

Infinitesimal rigidity implies local rigidity, but the examples of \cite{CLTGT} and \cite{CLTEM} show that the converse is not true.

We are working with aspherical manifolds, so computing the cohomology of a manifold or of its fundamental group does not make any difference.

For cusped manifolds one has a similar definition. Let $M$ denote a compact three manifold with boundary a union of tori 
and whose interior is hyperbolic with finite volume.

\begin{definition}
The manifold  $M$ is called \emph{infinitesimally projectively rigid relative to the cusps} if the inclusion $\partial M\subset M$ induces an injective homomorphism
$$
0\to H^1(M;\mathfrak{sl}(4)_{\Ad\rho})\to H^1(\partial M;\mathfrak{sl}(4)_{\Ad\rho}).
$$
\end{definition}

The following theorem provides infinitely many examples of infinitesimally projectively rigid $3$--dimensional manifolds.

\begin{thm}
\label{thm:fillingstrong}
Let $M$ be a compact orientable 3--manifold whose interior is hyperbolic with finite volume.
If $M$ is infinitesimally projectively rigid relative to the cusps, then infinitely many Dehn fillings on $M$ are infinitesimally projectively rigid.
\end{thm}

In his notes \cite{ThurstonNotes} Thurston defines the \emph{hyperbolic Dehn filling space}. 
He uses this space to prove that,  for all but a finite number of filling slopes on each boundary component, the
$3$--manifolds obtained by Dehn filling on $M$ are hyperbolic. The hyperbolic Dehn filling space is a subset  of the generalized hyperbolic Dehn filling coefficients 
and it is described in   Definition~\ref{def:hyp_dehn_filling_space} below, cf.\  \cite{NZ}.
The methods of Theorem \ref{thm:fillingstrong},
give the following:
\begin{thm}
\label{thm:fromonetomany}
Let $M$ be a compact orientable 3--manifold whose interior is hyperbolic with cusps. 
If a hyperbolic Dehn filling $N$ on $M$ satisfies:
\begin{itemize}
\item[(i)] $N$ is infinitesimally projectively rigid,
\item[(ii)] the Dehn filling slope of $N$ is contained in the (connected) hyperbolic Dehn filling space of $M$,
\end{itemize}
then infinitely many Dehn fillings on $M$ are infinitesimally projectively rigid.
\end{thm}

By Hodgson and Kerckhoff estimation of the size of the hyperbolic Dehn filling space \cite{HodgsonKerckhoff}, 
in a one cusped manifold all but at most 60 
topological Dehn fillings have a hyperbolic structure that lies in the hyperbolic Dehn filling space. Hence:

\begin{cor}
\label{cor:sixtyinfty}
Let $M$ be a one cusped hyperbolic manifold of finite volume. If 61 Dehn fillings on $M$ are either
non-hyperbolic or infinitesimally projectively rigid, then
infinitely many fillings are so.
\end{cor}

Those results are proved using the fact that all 
parameters of Thurston's hyperbolic Dehn filling space corresponding to non infinitesimally projectively rigid fillings on $M$ are contained in a proper analytic subset of the Dehn filling space, provided $M$ itself is infinitesimally projectively rigid.
This technique goes back to Kapovich in the setting of deformations of lattices of $PSO(3,1)$ in $PSO(4,1)$ \cite{Kapovich}.

Moreover, we obtain explicit examples of infinite families of infinitesimally projectively rigid manifolds. The Dehn filling parameters of these families lie on certain real analytic curves, and a careful analysis of the infinitesimal deformations of the corresponding manifolds results in the following proposition:
\begin{prop}
\label{prop:fig8explicit}
For a  sufficiently large positive integer $n$, the  homology sphere obtained by $1/n$--Dehn filling on the figure eight knot is  infinitesimally projectively rigid.

In fact, for every $k\in\mathbf Z$, $k\neq 0$, there exists $n_{k}>0$ such that if $n\geq n_k$ then the 
$k/n$--Dehn filling on the figure eight knot is  infinitesimally projectively rigid.
\end{prop}

{Theorem~\ref{thm:fillingstrong} provides infinitely many rigid Dehn fillings. On can ask whether there are still infinitely many non-rigid Dehn fillings. 
Though we do not have an example for manifolds, the following proposition shows that there are infinitely many non-rigid orbifolds obtained by Dehn 
fillings on the cusped manifold that satisfies the hypothesis of Theorem~\ref{thm:fillingstrong}.
\begin{prop}\label{prop:nonrigid}
The orbifold $\mathcal O_n$ with underlying space $S^3$, singular locus the figure eight knot and ramification index $n$ is not locally projectively rigid for sufficiently large $n$.
More precisely, its deformation space is a curve.
\end{prop}}

For any $n\in\mathbf N$, the Fibonacci manifold $M_n$ is the  cyclic cover of order $n$ of the orbifold $\mathcal O_n$ in Proposition~\ref{prop:nonrigid} \cite{HKM}.
 Hence $M_n$ is not projectively rigid, as deformations of the projective structure of $\mathcal O_n$ induce  deformations of $M_n$. There is an abundant literature about
 those manifolds. For instance, $M_4$ is not Haken but $M_n$ is Haken for $n\geq 5$, and Scannell has proved that they are not infinitesimally rigid in $SO(4,1)$ \cite{Scannell}.

Using that punctured torus bundles with tunnel number one are obtained by 
$n$--Dehn filling on the Whitehead link (cf.\ \cite{Akiyoshi}), we shall prove:

\begin{prop}
\label{prop:twist}
All but finitely many punctured torus bundles with tunnel number one are infinitesimally projectively rigid relative to the cusps.

All but finitely many twist knots complements are infinitesimally projectively rigid relative to the cusps.
\end{prop}

\medskip

The real hyperbolic space $\mathbf H^3$ naturally embeds in the
complex hyperbolic space $\mathbf H_{\mathbf C}^3$.
We may study  the corresponding  deformation theory 
coming from viewing  $PSO(3,1)= \operatorname{Isom}^+(\mathbf H^3)$ in $PSU(3,1)=\operatorname {Isom}_0(\mathbf H^3_{\mathbf C})$, i.e. the identity component of complex hyperbolic isometries.

\begin{definition}
We say that $M$ is \emph{infinitesimally  $\mathbf H^3_{\mathbf C}$--rigid} relative to the cusps if the sequence
$$
0\to H^1(M;\mathfrak{su}(3,1)_{\Ad\rho})\to H^1(\partial M;\mathfrak{su}(3,1)_{\Ad\rho})
$$
is exact
\end{definition}

In particular, if $\partial M=\emptyset$, then we require $H^1(M;\mathfrak{su}(3,1)_{\Ad\rho})=0$.
The study of deformations in $PGL(4)$ and $PSU(3,1)$ are related, as we shall see in 
Subsection~\ref{subsectioncomplex}. In particular we have the following theorem of
Cooper, Long and Thistlethwaite.

\begin{thm}\cite{CLTGT}
\label{thm:projectivehyperbolic}
Let $M^n$ be a real hyperbolic manifold of finite volume, $n\geq 3$. 
Then $M^n$ is infinitesimally projectively rigid relative to the cusps
if and only if $M^n$ is infinitesimally $\mathbf H^n_{\mathbf C}$--rigid  relative to the cusps.
\end{thm}

This equivalence is described by means of Lie algebras, and it is used along the paper, because some things are
easier to understand in the complex hyperbolic setting instead of the projective one.

\medskip

The article is organized as follows. In Section~\ref{sec:Thurstonsslice} we recall Thurston's construction
of deformations of hyperbolic structures and the generalized Dehn filling coefficients.
In Section~\ref{sec:infinitesimaldefs} we introduce the main tools in order to study infinitesimal deformations. 
The next two sections are devoted to cohomology computations, namely in Section~\ref{sec:invariantsubspaces}
we compute invariant subspaces of the Lie algebras and in Section~\ref{sec:cohmologytorus} we analyze the image in cohomology
of the restriction to the torus boundary. The proof of Theorems~\ref{thm:fillingstrong} and \ref{thm:fromonetomany}
 is given in Section~\ref{sec:function},
by means of an analytic function on the deformation space: when this function does not vanish, then the corresponding 
Dehn filling is infinitesimally rigid. To prove Propositions~\ref{prop:fig8explicit} and \ref{prop:twist}, we require 
the notion of rigid slope, treated in Section~\ref{sec:flexingslopes}, as well as explicit computations on the figure eight 
knot and the Whitehead link exteriors, made in Section~\ref{sec:examples}.

\paragraph{Acknowledgements} We are indebted to Suhyoung Choi for   useful conversations, as well to the anonymous referee(s)
for suggesting many improvements.

\section{Dehn fillings and Thurston's slice}

\label{sec:Thurstonsslice}

In this section we recall the deformation space introduced by Thurston in his proof of hyperbolic Dehn filling theorem
\cite{ThurstonNotes}.

Along the paper, $M$ denotes a compact manifold with  boundary a  union  of $k>0$ tori and hyperbolic interior: 
$$
\partial M=\partial_1 M\sqcup\cdots\sqcup \partial_k M,
$$
where each $\partial_i M\cong T^2$.

 The deformation space of hyperbolic structures of $M$ around the complete structure is described by the 
\emph{Thurston's slice} \cite{ThurstonNotes,NZ}. 
Given $\lambda_i,\mu_i\in\pi_1(\partial M)$ a pair of simple closed curves that generate the fundamental group on each component $\partial M_i$,
Thurston introduced a parameter 
$$
u=(u_1,\ldots,u_k) \in U\subset \mathbf C^k,
$$ 
defined on $U$ a neighborhood of $0$. The neighborhood $U$  parametrizes the deformations of the complete holonomy of the interior of $M$.
Two structures parametrized by $u$ and $u'\in U$ are equivalent 
(the developing maps differ by composing with 
an isometry of $\mathbf H^3$) if and only if 
\begin{equation}
\label{eqn:samecharacter}
(u_1,\ldots,u_k)= (\pm u_1',\ldots, \pm u_k') .
\end{equation}
This is a consequence of the fact that (\ref{eqn:samecharacter}) is a criterion for having the same character, 
and the fact that deformations are parametrized by conjugacy classes of holonomy \cite{CEG}.

\begin{thm}[Thurston's slice]
\label{thm:thurston}
There exists an open neighborhood $0\in U\subset \mathbf C^k$, an analytic family of representations $\{\rho_u\}_{u\in U}$,
of $\pi_1(M)$ in $PSL_2(\mathbf C)$ and analytic functions $v_i=v_i(u)$, $i=1,\ldots,k$  so that:
\begin{enumerate}[(i)]
\item The parameters $u_i$ and $v_i$ are the complex length of $\rho_u(\mu_i) $ and $\rho_u(\lambda_i)$ respectively.
\item The function $\tau_i(u)=v_i(u)/u_i$ is analytic.
Moreover $v_i=\tau_i(0) u_i + (\vert u\vert ^3)$, where $\tau_i(0)\in\mathbf C$ is the cusp shape and has nonzero imaginary part. 
\item The structure with holonomy $\rho_u$ is complete on the $i$th cusp if and only if  $u_i=0$. 
\item  When $u_i\neq 0$, the equation 
\begin{equation}
\label{eqn:Dehnfilling}
p_i\, u_i+q_i\, v_i=2\pi \, \mathbf  i
\end{equation}
has a unique solution $(p_i,q_i)\in\mathbf R^2$. 
The representation $\rho_u$ is the holonomy of a incomplete hyperbolic structure with \emph{generalized Dehn filling coefficients} $(p_i,q_i)$ on the $i$th cusp.
\end{enumerate}
\end{thm}

See \cite[App.~B]{BoileauPorti} for a proof, for instance. 

In his proof of hyperbolic Dehn filling, Thurston shows that there is a diffeomorphism between $U$ and a neighborhood of $\infty$ in 
$(\mathbf R^2\cup\{\infty\})^k$ that
maps componentwise $0$ to $\infty$ and $u_i\neq 0$ to $(p_i,q_i)\in\mathbf R^2$ satisfying $p_i\, u_i+q_i\, v_i=2\pi \, \mathbf  i$.

\begin{definition}
\label{def:hyp_dehn_filling_space}
The connected neighborhood of $\infty$ in 
$(\mathbf R^2\cup\{\infty\})^k$ that is the  image of $U$ is called the \emph{hyperbolic Dehn filling space}.
\end{definition}

The geometric interpretation of   generalized Dehn filling coefficients is the following one:
the representation $\rho_0$ is the holonomy of the complete hyperbolic structure of  $\mathit{Int}(M)$. 
If $u\neq0$ then
the representation $\rho_u$ is the holonomy of a non complete hyperbolic structure $M_u$ on $\mathit{Int}(M)$ and the metric completion of $M_u$ is described by the Dehn filling parameters 
\begin{enumerate}[(i)]
\item When $p_i,q_i\in\mathbf Z$ are coprime, then the metric completion of $M_u$ is precisely the Dehn filling with slope $p_i\mu_i+q_i\lambda_i$.
\item  When $p_i/q_i=p'_i/q'_i\in\mathbf Q\cup\infty$ with $p'_i,q'_i\in\mathbf Z$ coprime, then the completion of $M_u$ is a cone manifold,
obtained by   Dehn filling with slope $p'_i\mu_i+q'_i\lambda_i$ where the core of the torus is a singular geodesic with cone angle $2\pi p'_i/p_i$.
\item When $p_i/q_i\in \mathbf R\setminus\mathbf Q$, then the metric completion is the one point compactification.
\end{enumerate}

A particular case that we will use later is when $u_i=\alpha_i\, \mathbf i$ for some $\alpha_i\in \mathbf R$, $\alpha_i>0$. Then $p_i=2\pi/\alpha_i$ and $q_i=0$,
and
$\rho_{(\mathbf i\alpha_1,\ldots ,\mathbf i\alpha_k)}$ is the holonomy of a hyperbolic cone manifold with cone angles 
$(\alpha_1,\ldots ,\alpha_k)$.

The real analytic structure will be crucial in our arguments. When viewed in $PSL_2(\mathbf C)$, $\rho_u$ is complex analytic, but we will work with the real analytic structure, which is the same as for
$PSO(3,1)$. In particular the following lemma will be useful.

\begin{lemma}\label{lem:analytic}
For each $i=1,\ldots,k$, if $\tau_i(u)=v_i(u)/u_i$, then the map
$$
\begin{array}{rcl}
U\subset\mathbf C^k & \to & \mathbf R^2 \\
u & \mapsto & \frac 1{\vert p_i+q_i\tau_i\vert ^2}(p_i,q_i)
\end{array}
$$ 
is real analytic.
\end{lemma}

\begin{proof}
Using Equation~(\ref{eqn:Dehnfilling}), we obtain:
$$
p_i=-2\pi\frac{Re(u_i\tau_i)}{\vert u_i\vert^2 Im(\tau_i)}, 
\qquad
q_i=2\pi\frac{Re(u_i)}{\vert u_i\vert^2 Im(\tau_i)},
\qquad
p_i+q_i \tau_i=\frac{2\pi\mathbf i}{u_i}.
$$
The lemma is a straightforward consequence from these equalities and the fact that the 
imaginary part of $\tau_i(0)$ does not vanish.
\end{proof}

\section{Infinitesimal deformations}
\label{sec:infinitesimaldefs}

The aim of this section is to provide some technical background for the sequel and to set up the notations.
In Subsection~\ref{subsec:cohom} we shall recall the setup of twisted homology theory,
Subsection~\ref{subsection:so(3,1)} provides some known results about the cohomology group 
$H^1(M; \mathfrak{so}(3,1)_{\Ad\rho})$ 
(Subsection~\ref{subsection:so(3,1)}).

\subsection{The cohomology, Kronecker pairings and the Poincar\'e-Lefschetz duality}
\label{subsec:cohom}

\subsubsection{The homology and cohomology with twisted coefficients}
\label{subsection:twistedcohomology}
Let $X$ be a finite CW-complex, let $V$ be a finite dimensional real vector space and let $\rho\co\pi_1(X)\to GL(V)$ be a representation.
In the sequel of this subsection we shall denote $\widetilde X$ the universal covering of $X$ and $\pi := \pi_1(X)$ for short its fundamental group.
The vector space $V$ and its dual $V^*$ turn into a left modules over the group ring 
$\mathbf Z \pi$: for all $\gamma\in \pi$, $v\in V$ and $f\in V^*$ we have 
\[ \gamma\, v := \rho(\gamma) v \text{ and } \gamma\, f( v) = f(\rho(\gamma)^{-1} v)\,.\]

The homology and cohomology of $X$ with coefficients in $V$ are defined in the usual way:
\begin{align*}
C_*(X;V) &:= C_*(\widetilde X)\otimes_{\mathbf Z \pi} V\\
C^*(X;V^*) &:= \hom_{\mathbf Z\pi} (C_*(\widetilde X); V^*)\,.
\end{align*}
Here we follow the standard notations and conventions (see \cite[3.H]{Hatcher}).
The boundary and coboundary operators are given by
\begin{align*}
\partial_p &= \partial \otimes \Id \co C_p(X;V) \to C_{p-1}(X;V)\: ;\\
\delta^p &\co C^{p-1}(X;V) \to C^{p}(X;V), \quad \delta^p F (c_p) = F(\partial c_p)
\end{align*}
where $\partial$ denotes the boundary operator of $C_*(\widetilde X)$.
Note that $C_*(X;V)$ and $C^*(X;V^{*})$ are finite dimensional vector spaces due to the finiteness of $X$.

\subsubsection{The group cohomology and infinitesimal deformations}
\label{subsection:groupcohomology}

 Let $\pi=\pi_1(X)$, $V$ and $\rho:\pi\to GL(V)$ be as in the previous paragraph. The group cohomology of
$\pi$ with coefficients in $V$ is denoted by
$$
H^*(\pi; V).
$$
See for instance \cite{Bro82} for definitions and proofs of this paragraph.
We are mainly interested in the case where $X$ is a hyperbolic manifold, hence aspherical. Thus we have a natural isomorphism:
$$
H^*(\pi; V)\cong H^*(X; V).
$$
(see \cite{Whi78} for details).
For the interpretation of $H^1(\pi; V)$ in terms of infinitesimal deformations, we need to recall
 that the space of 1-cocycles or crossed morphisms is
$$
Z^1(\pi;V)=\{ d:\pi\to V\mid d(\gamma_1\gamma_2)=d(\gamma_1)+\rho(\gamma_1) d(\gamma_2),\ \forall\gamma_1\gamma_2\in \pi\}.
$$
The space of coboundaries is
$$
B^1(\pi;V)=\{d\in Z^1(\pi;V)\mid \exists v\in V \textrm{ s.t. } d(\gamma)=(\rho(\gamma)-1) v ,\ \forall\gamma_1\gamma_2\in \pi\}.
$$
Then we have a natural isomorphism
$$
H^1(\pi;V)\cong Z^1(\pi;V)/B^1(\pi;V).
$$
Weil's construction \cite{Weil} gives the cohomological interpretation of infinitesimal deformations. Namely, given $G$ a Lie group
with Lie algebra $\mathfrak{g}$ and a representation $\psi\co\pi\to G$, the space of infinitesimal deformations is
$$
H^1(\pi; \mathfrak{g}_{\Ad\psi}),
$$
where  $\Ad\psi$ denotes the composition of $\psi$ with the adjoint representation, hence it is a representation of $\pi$ into    $GL(\mathfrak{g})$.
The construction of Weil is as follows. A  deformation of $\psi\co\pi\to G$ may be written as $\psi_t\co\pi\to G$, where $\psi_0=\psi$ and $t\in (-\varepsilon,\varepsilon)$.
Assuming differentiability at $t=0$, then define $d\co\pi\to \mathfrak{g}$ by
$$
d(\gamma):=\left.\frac{d\phantom{t}}{dt} \right|_{t=0}\psi_t(\gamma) \, \psi_0(\gamma)^{-1}\qquad \forall\gamma\in \pi.
$$
It is easy to check that $d\in Z^1(\pi;\mathfrak{g}_{\Ad\psi})$. In fact Weil proves that $Z^1(\pi;\mathfrak{g}_{\Ad\psi})$ is precisely the space of infinitesimal (or first order) deformations
of $\psi$, 
and that $B^1(\pi;\mathfrak{g}_{\Ad\psi})$  is the space of infinitesimal deformations by conjugation \cite{Weil}.

\begin{thm}[Weil \cite{Weil}]
Infinitesimal rigidity implies actual rigidity: If $H^1(\pi; \mathfrak{g}_{\Ad\psi})=0$ then  $ \psi$ can not be deformed up to conjugation
\end{thm}

\subsubsection{The Kronecker pairing}
\label{subsection:Kronecker}

 Let $X$, $\pi=\pi_1(X)$, $V$ and $\rho\co\pi\to GL(V)$ be as in the first paragraph.
There is a natural isomorphism
\[ \tau\co C_*(X ; V)^* \to C^*(X ; V^*) = \hom_{\mathbf Z\pi} ( C_*(\widetilde X) ; V^* )\,. \]
(see \cite[2.6]{Weibel}). For  $g\in C_p(X ; V)^*$,
$G\in C^p(X ; V^*)$, $c_p\in C_p(\widetilde X)$ and $v\in V$ the isomorphism $\tau$ and its inverse are given 
\[ \tau g (c_p) (v) = g(c_p\otimes v) \text{ and } \tau^{-1}G ( c\otimes v) = G(c) (v)\,.\] 
This gives rise to the \emph{Kronecker pairing}
\[  \langle \, .\, , \, .\,\rangle\co C^p(X ; V^*)\otimes C_p(X ; V) \to \mathbf R\]
given by $\langle G, c_p\otimes v \rangle = G(c)(v)$.
The Kronecker pairing  behaves well with respect of the boundary operators:
\[  \langle G, \partial_p (c_p\otimes v) \rangle = 
\langle   \delta^p G, c_p\otimes v \rangle\,. \]
This shows that we obtain a well defined pairing 
\[ \langle \, .\, , \, .\,\rangle\co H^p(X ; V^*)\otimes H_p(X ; V) \to \mathbf R\,.\]
In order to see that the Kronecker pairing is non degenerated we observe that the boundary operators $\partial_p$ and $\delta^p$ are dual to each other and hence 
$\Ker \partial_p = (\image \delta^p)^\bot$ and 
$\Ker \delta_p = (\image \partial^p)^\bot$ (see \cite[2.26]{Greub}).
Therefor if $F$ is a $p$--cocycle and if $\langle F, z_p \rangle =0$ for every
$z_p\in \Ker \partial_p$ then $F\in (\Ker\partial_p)^\bot = \image\delta^p$ and hence $F$ is a coboundary.

Now let $b\co V\times V\to\mathbf R$ be a non-degenerate  bilinear form.
Moreover we shall assume that $b$ is $\pi$--invariant i.e.\ for all $\gamma\in\pi$ and for all $v,w\in V$ we have
\[ b(v,w) = b(\gamma\,v , \gamma\,w)\,.\]
The form $b$ induces an isomorphism $\phi_b$ between the $\pi$--modules $V$ and $V^*$:
for $v,w\in V$ we have $\phi_b(v) (w) = b(v,w)$. The map $\phi_b$ is injective and hence an isomorphism since $b$ is non-degenerate. Observe that $\phi_b$ is $\pi$--invariant:
\[ \phi_b(\gamma\, v) (w) = b(\gamma\, v, w)= b( v,\gamma^{-1} w) = \phi_b(v)(\gamma^{-1} w)=\gamma\phi_b(v) (w)\,.\]
Now $b$ gives rise to a Kronecker pairing
\[\langle \, .\, , \, .\,\rangle\co C^p(X ; V)\otimes C_p(X ; V) \to \mathbf R\]
given by $\langle F, c_p\otimes v \rangle = \phi_b(F(c_p))(v) = b(F(c_p),v)$ for $F\in C^p(X;V)$ and 
$c_p\otimes v\in C_p(X;V)$. Hence we obtain a non-degenerate form
\[ \langle \, .\, , \, .\,\rangle\co H^p(X ; V)\otimes H_p(X ; V) \to \mathbf R\,.\]

\subsubsection{The Poincar{\'e}-Lefschetz duality}
Let $M$ be a compact, oriented, $n$--dimensional manifold with boundary $\partial M$ and let
$\rho\co\pi_{1}(M)\to GL(V)$ be a representation.
The intersection number between simplices of two dual triangulations on the universal covering 
$\widetilde M$ induces the perfect intersection pairing
\[ H_{p}(M;V^{*})\otimes H_{n-p} (M,\partial M ; V)\to \mathbf R \]
(see \cite[0.3]{Porti}, \cite[Sec.~4]{JohnsonMillson},\cite[Lemma 2]{Milnor62} and for a general approach \cite{Steenrod}).

Hence we obtain an isomorphism between $H_{p}(M;V^{*})^{*}$ and $H_{n-p} (M,\partial M ; V)$. Composing this isomorphism with the isomorphism obtained from the Kronecker pairing we obtain the duality isomorphism
\[ \mathit{PD}\co H_{n-p} (M,\partial M ; V)\to H^{p}(M;V)\,.\]
An isomorphism $\mathit{PD}\co H_{p} (M ; V^{*})\to H^{n-p} (M,\partial M ; V^{*})$ is obtained analogously. The usual formula for the cup-product (to be denoted $\cup$) of simplicial  cochains gives  that the cup-product induces a perfect pairing
\[ .\,\cup \,.\co H^{n-p} (M,\partial M ; V)\otimes H^{p}(M;V^{*})\to \mathbf R\,. \]
Moreover the existence of a non-degenerate bilinear map $b\co V\times V\to\mathbf R$ gives a pairing
\[ .\,\overset{b}{\cup} \,.\co H^{n-p} (M,\partial M ; V)\otimes H^{p}(M;V) \to \mathbf R\,. \]

\subsubsection{Killing forms}
The \emph{Killing form} on any Lie algebra $\mathfrak{g}$ is defined as:
$$
B(X,Y)= \operatorname{trace}( ad_X \circ ad_Y ) 
\qquad\forall X,Y \in\mathfrak{g},
$$
where   $ad_X \in \mathrm{End}(\mathfrak{g})$ denotes  the endomorphism given by
$ad_X(Y)= [X,Y]$. If $\mathfrak{g} = \mathfrak{sl}(4)$, then 
$B(X,Y)= 8 \tr(X\cdot Y)$. Note that $B$ is symmetric.

The matrix of the Lorentzian inner product is denoted by
$$J=\begin{pmatrix}
       1 &   &   &   \\
          & 1 &   &   \\
	   &   & 1 &   \\
          &   &   & -1
      \end{pmatrix}.
$$
So that
$$
O(3,1)=\{ A\in GL(4)\mid A^{t}JA=J\},
$$
and the connected component of the identity of its projectivization $PSO(3,1)$ is the group of orientation preserving isometries of $\mathbf H^3$.
Its Lie algebra is
$$
   \mathfrak{so}(3,1)=\{ a\in \mathfrak{sl}(4) \mid a^{t} J= -J a\}.
$$

Following Johnson and Millson \cite{JohnsonMillson}, along the paper we shall use the decomposition of $\mathfrak{sl}(4)$ as direct sum  of $PSO(3,1)$--modules via the adjoint action:
\begin{equation}
\label{eqn:sl}
\mathfrak{sl}(4)= \mathfrak{so}(3,1)\oplus \mathfrak{v},
\end{equation}
where 
$$
\mathfrak{v}=\{ a\in \mathfrak{sl}(4) \mid a^{t} J= J a\}.
$$
Notice that $\mathfrak v$ is not a Lie algebra, but just a $PSO(3,1)$--module, and that  
$$
\dim(\mathfrak{v})=\dim(\mathfrak{gl}(4))- \dim(\mathfrak{so}(3,1))=15-6=9.
$$

Both the form $B$ on $\mathfrak{sl}(4)$ and its restriction to $\mathfrak{so}(3,1)$
are non-degenerate.
Moreover $\mathfrak{v}$ is the orthogonal complement to $\mathfrak{so}(3,1)$: 
$$
\mathfrak{sl}(4)=\mathfrak{so}(3,1)\perp\mathfrak{v}.
$$
Therefore $B$ 
restricted to $\mathfrak{v}$ is non-degenerate, too.

Hence given a representation $\rho\co\pi_1(X)\to PSO(3,1)$ we obtain a \emph{canonical splitting} in homology:
$$
H^*(X; \mathfrak{sl}(4)_{\Ad\rho})= H^*(X; \mathfrak{so}(3,1)_{\Ad\rho})\oplus H^*(X; \mathfrak{v}_{\Ad\rho}).
$$

In the sequel we shall be mostly interested in the following situation:
let $M$ be a finite volume hyperbolic 3-manifold with $k$ cusps and let
$\rho\co\pi_{1}(M)\to SO(3,1)\subset SL(4)$ be a representation. 
Then the Lie algebra $\mathfrak{sl}(4)$ 
turns into a $\pi_{1}(M)$--module via $\Ad\circ\rho\co \pi_{1}(M)\to\Aut{\mathfrak{sl}(4)}$.
Note that the Killing form is $\pi_{1}(M)$ invariant hence the action of $\pi_{1}(M)$ respects
$\mathfrak{v}$ i.e.\ the 9-dimensional vector space $\mathfrak{v}$ turns into a $\pi_{1}(M)$--module and the restriction of the Killing form on $\mathfrak{v}$ induces a non degenerate $\pi_{1}(M)$ invariant bilinear form
\[ B\co \mathfrak{v}\times\mathfrak{v}\to \mathbf R\,.\]

A \emph{cup product} on cohomology is defined by using $B$:
\begin{equation}
\label{eqn:cupproduct}
H^p(M;\mathfrak{v})\otimes H^q(M,\partial M;\mathfrak{v})\xrightarrow{\cup} H^{p+q}(M,\partial M;\mathfrak{v}\otimes \mathfrak{v})\xrightarrow{B_*} H^{p+q}(M,\partial M;\mathbf R)
\end{equation}
where the first arrow is the usual cup product, and
$B_*$ denotes the map induced by 
$B\co \mathfrak{v}\otimes \mathfrak{v}\to \mathbf R$.
In the sequel this cup product will be simply denoted by $\cup$.

This cup product induces Poincar\'e-Lefschetz duality 
\[H^p(M;\mathfrak{v})\otimes H^{n-p}(M,\partial M;\mathfrak{v})\xrightarrow{\cup} H^{n}(M,\partial M;\mathbf R)\cong\mathbf R\]
since $B$ is non degenerated.
As $B$ is symmetric, this cup product is symmetric or antisymmetric depending on whether
the product of dimensions $p\,q$ is even or odd, as the usual cup product.

\subsection{The semi-continuity}

Let $V$ be a finite dimensional real vector space and let $\Gamma$ be a finitely generated group.
The set $R(M,GL(V))$ of all representations of $\Gamma$ into $GL(V)$ has the structure of a real affine algebraic set, $R(\Gamma,GL(V))\subset \mathbf R^{N}$ (see \cite{LubotzkyMagid}). Given a representation $\rho\co\Gamma\to GL(V)$ the vector space $V$ turns into a 
$\Gamma$--module via $\rho$ which will be denoted by $V_{\rho}$.  

\begin{lemma}
 \label{lemma:uppersc} Let $V$ be a finite dimensional real vector space.
Then the function $h^{i}\co R(\Gamma,GL(V)) \to \mathbf Z$ given by 
\[h^{i}(\rho) = \dim H^i(\Gamma; V_{\rho}) \]
is upper semi-continuous for $i=0,1$.

More precisely, for each $n\in\mathbf Z$ the set
$\{\rho\in R(\Gamma,GL(V)) \mid h^{i}(\rho) \geq n\}$ is a closed algebraic subset.
\end{lemma}

\begin{proof}
This follows from a general theorem \cite[Theorem 12.8]{Hartshorne}.
For the convenience of the reader we give an elementary argument.

We view $H^1(\Gamma;V_{\rho})$ as the group cohomology of 
$\Gamma$, namely it is isomorphic to the quotient 
$Z^1(\Gamma;V_{\rho})/B^1(\Gamma;V_{\rho})$. The space of cocycles 
$Z^1(\Gamma;V_{\rho})$ is the kernel of a linear map with coefficients that are polynomial functions in the ambient coordinates, hence 
$\dim( Z^1(\Gamma;V_{\rho}) )$ is constant, except on a (possibly empty) closed algebraic subset where it is larger.
  On the other hand, the space of coboundaries $B^1(\Gamma;V_{\rho})$
is the image of a linear map whose coefficients are polynomial functions in the ambient coordinates, 
hence 
$
  \dim (B^1(\Gamma;V_{\rho})) 
$  is constant, except  on a closed algebraic subset where it is smaller.

Analogously, $H^0(\Gamma;V_{\rho})$ is the kernel of a linear map with coefficients that are polynomial functions in the ambient coordinates. Hence 
$\dim( H^0(\Gamma;V_{\rho}) )$ is constant, except on a (possibly empty) closed algebraic subset where it is larger.
\end{proof}


\subsection{Infinitesimal deformations in real hyperbolic spaces}
\label{subsection:so(3,1)}

Infinitesimal deformations in 
$\operatorname{Isom}^+(\mathbf H)=PSO(3,1)$ are well understood, and described by
$H^1(M; \mathfrak{so}(3,1)_{\Ad\rho})$. We summarize in this subsection the main results.

Let $M$ be a finite volume hyperbolic 3-manifold with $k$ cusps. We choose one essential simple closed curve $\mu_i\subset \partial_i M$ for each boundary component.

\begin{prop} Let $M$ be a finite volume hyperbolic 3-manifold with $k$ cusps and 
$\mu=\mu_1\cup\cdots \cup \mu_k \subset\partial M$ be given as above. Moreover 
let $U$ and $\{\rho_{u}\}_{u\in U}$ be as in Theorem~\ref{thm:thurston}.

Then there exists a possibly smaller open neighborhood $\tilde U$ of $0\in \tilde U\subset U$ such that
for all $u\in \tilde U$:
\begin{enumerate}[(i)]
\item The inclusion $\partial M\subset M$ induces a monomorphism
$$
0\to H^1(M; \mathfrak{so}(3,1)_{\Ad\rho_u})\to H^1(\partial M; \mathfrak{so}(3,1)_{\Ad\rho_u}).
$$ 
\item The inclusion of the union $\mu=\mu_1\cup\cdots \cup \mu_k  \subset M$ induces a monomorphism
$$
0\to H^1(M; \mathfrak{so}(3,1)_{\Ad\rho_u})\to H^1(\mu; \mathfrak{so}(3,1)_{\Ad\rho_u}).
$$ 
\item $\dim 
H^1(M; \mathfrak{so}(3,1)_{\Ad\rho_u})=2 k$.
\item $\dim 
H^1(M,\mu ; \mathfrak{so}(3,1)_{\Ad\rho_u})=2 k$.
\end{enumerate}
The $\tilde U$ depends on $\mu$ for (ii) and (iv), but not for (i) and (iii).
\end{prop}

\begin{proof}
First note that for $u=0$ the representation $\rho_{0}$ is the holonomy of the complete hyperbolic structure.
This implies that $H^{0}(M;\mathfrak{so}(3,1)_{\Ad\rho_0})=0$, 
$H^{0}(\mu;\mathfrak{so}(3,1)_{\Ad\rho_0})\cong \mathbf R^{2k}$,
$H^{0}(\partial M;\mathfrak{so}(3,1)_{\Ad\rho_0})\cong \mathbf R^{2k}$,
$$\dim(H^{1}(M;\mathfrak{so}(3,1)_{\Ad\rho_0}))= 2k \text{, and }
\dim(H^{1}(M,\mu;\mathfrak{so}(3,1)_{\Ad\rho_0}) )= 2k$$
(see \cite[8.8]{KapovichBook}, \cite{Garland} or   also \cite{Bromberg04} and \cite{MenalP} for a proof).

Now by the semi-continuity, we can choose $\tilde U$ such that for all $u\in \tilde U$
$H^{0}(M;\mathfrak{so}(3,1)_{\Ad\rho_u})=0$, $H^{0}(\mu;\mathfrak{so}(3,1)_{\Ad\rho_u})\cong \mathbf R^{2k}$, $H^{0}(\partial M;\mathfrak{so}(3,1)_{\Ad\rho_u})\cong \mathbf R^{2k}$
$$\dim(H^{1}(M;\mathfrak{so}(3,1)_{\Ad\rho_u}))\leq 2k \text{ and } \dim(H^{1}(M,\mu;\mathfrak{so}(3,1)_{\Ad\rho_u}) )\leq 2k.$$
Here we used that for every representation 
$\rho\co\pi_{1}(M)\to PSO(3,1)$ we have $H^{0}(\mu_{i};\mathfrak{so}(3,1)_{\Ad\rho})\cong \mathbf R^{2}$ if and only if $\rho(\mu_{i})\neq 1$.

Next the long exact sequence of the pair $(M,\mu)$ is given by
\begin{multline*} 0\to H^{0}(\mu;\mathfrak{so}(3,1)_{\Ad\rho_u}) \to H^{1}(M,\mu;\mathfrak{so}(3,1)_{\Ad\rho_u})\to \\ H^{1}(M;\mathfrak{so}(3,1)_{\Ad\rho_u}) \to
H^{1}(\mu;\mathfrak{so}(3,1)_{\Ad\rho_u})\to \cdots \,.
\end{multline*}
Therefore for all $u\in \tilde U$ the map
$H^{0}(\mu;\mathfrak{so}(3,1)_{\Ad\rho_u}) \to H^{1}(M,\mu;\mathfrak{so}(3,1)_{\Ad\rho_u})$ is injective and hence surjective
since $\dim(H^{1}(M,\mu;\mathfrak{so}(3,1)_{\Ad\rho_u}) )\leq 2k$. It follows from this that for all $u\in\tilde U$ 
$\dim(H^{1}(M,\mu;\mathfrak{so}(3,1)_{\Ad\rho_u}) )= 2k$ and that the map induced by the inclusion $\mu\subset M$ gives a monomorphism
\[0\to H^{1}(M;\mathfrak{so}(3,1)_{\Ad\rho_u}) \to
H^{1}(\mu;\mathfrak{so}(3,1)_{\Ad\rho_u}) \,.\]

Since the inclusion $\mu\subset M$ factors through the $\partial M$, $\mu\subset \partial M\subset M$, we obtain that the map $H^{1}(M;\mathfrak{so}(3,1)_{\Ad\rho_u}) \to H^{1}(\mu;\mathfrak{so}(3,1)_{\Ad\rho_u})$ 
factors through $H^{1}(M;\mathfrak{so}(3,1)_{\Ad\rho_u}) \to H^{1}(\partial M;\mathfrak{so}(3,1)_{\Ad\rho_u})$ hence for all $u\in \tilde U$ the map
\[ H^{1}(M;\mathfrak{so}(3,1)_{\Ad\rho_u}) \to H^{1}(\partial M;\mathfrak{so}(3,1)_{\Ad\rho_u}) \]
is injective. Hence 
\[ H^{0}(\partial M;\mathfrak{so}(3,1)_{\Ad\rho_u}) \to H^{1}(M,\partial M;\mathfrak{so}(3,1)_{\Ad\rho_u})\]
is an isomorphism. Moreover Poincar{\'e}--Lefschetz duality gives that
\[ H^{1}(M,\partial M;\mathfrak{so}(3,1)_{\Ad\rho_u}) \cong H^{2}(M;\mathfrak{so}(3,1)_{\Ad\rho_u})^{*}\]
and hence $\dim (H^{2}(M;\mathfrak{so}(3,1)_{\Ad\rho_u}) ) = 2k$. Now the Euler characteristic of $M$ is zero and $H^{0}(M;\mathfrak{so}(3,1)_{\Ad\rho_u})=0$ which implies 
$$\dim (H^{1}(M;\mathfrak{so}(3,1)_{\Ad\rho_u}) ) = 2k\,.$$
\end{proof}
\begin{remark}
This proposition can be seen as the algebraic part of 
Thurston's hyperbolic Dehn filling theorem.
\end{remark}

\subsection{Complex hyperbolic space}
\label{subsectioncomplex}

Consider $\mathbf C^{3,1}$ i.e.\ $\mathbf{C}^4$ with the 
hermitian product
$$
\langle w, z\rangle=   w_1 \bar z_1+  w_2 \bar z_2+  w_3 \bar z_3 
- w_4 \bar z_4 = w^t J \bar z =  z^* w
$$
where $z^*=\bar z^t J$.
Its projectivization $\mathbf P^{3,1} := \mathbf P(\mathbf C^{3,1})$ gives rise to complex hyperbolic space $\mathbf H^3_{\mathbf C }$. More precisely, $\mathbf H^3_{\mathbf C }=\{ [v]\in\mathbf P^{3,1}\mid \langle v, v\rangle <0\}$ 
cf.~\cite{Goldman,Epstein}. Here and in the sequel $[v]$ denotes the line generated by the non zero vector $v\in \mathbf C^{3,1}$.

Let
$$
SU(3,1)=\{A\in SL(4,\mathbf C)\mid \bar A^t J A= J \}.
$$
The group of holomorphic isometries of complex hyperbolic space is the projectivization
$PSU(3,1)=PU(3,1)$, with Lie algebra:
$$
\mathfrak{su}(3,1)=\{ a\in \mathfrak{sl}(4,\mathbf C)\mid \bar a^t J= -J a \}.
$$
The key point is that, as $SO(3,1)$--module, this Lie algebra has a decomposition:
\begin{equation}
\label{eqn:su}
\mathfrak{su}(3,1)=\mathfrak{so}(3,1)\perp \mathbf i\, \mathfrak{v}.
\end{equation}
Thus:

\begin{remark}
\label{remark:Vcomplex}
The subspace $\mathfrak{v}=\{ a\in\mathfrak{sl}(4)\mid a^t\,J=J\, a\}$ 
can be seen as the imaginary part of infinitesimal deformations in complex hyperbolic space.
\end{remark}

\begin{proof}[Proof of Theorem~\ref{thm:projectivehyperbolic}]
We define
$$
\mathfrak{v}_n=\{ a\in \mathfrak{sl}(n+1) \mid a^{t} J= J a\},
$$
where $J$ is the symmetric matrix with one negative and $n$ positive eigenvalues,
generalizing the definition of $\mathfrak{v}$ for $n=3$. The generalizations of Equations~(\ref{eqn:sl})
and (\ref{eqn:su}) are
$$
\mathfrak{sl}(n+1)=\mathfrak{so}(n,1)\perp  \mathfrak{v}_n
$$
and
$$
\mathfrak{su}(n,1)=\mathfrak{so}(n,1)\perp \mathbf i\, \mathfrak{v}_{n},
$$
that are isomorphisms of $SO(n,1)$--modules via the adjoint action.

Let $M^n$ denote a compact n-manifold whose interior has a finite volume hyperbolic structure, as in the statement of the theorem.
By Garland's infinitesimal rigidity \cite{Garland}, the map induced by inclusion
$$
H^1(M^n;\mathfrak{so}(n,1))\to H^1(\partial M^n;\mathfrak{so}(n,1))
$$
is injective (here all $SO(n,1)$--modules become $\pi_1(M^n)$--modules via the holonomy).
Thus $M^n$ is infinitesimally projectively rigid relative to the cusps if and only if
$$
H^1(M^n;  \mathfrak{v}_n)\to H^1(\partial M^n;  \mathfrak{v}_n   )
$$
is injective, and 
 $M^n$ is infinitesimally $\mathbf H^n_{\mathbf C}$--rigid relative to the cusps if and only if
$$
H^1(M^n;   \mathbf i \mathfrak{v}_n)\to H^1(\partial M^n;  \mathbf i \mathfrak{v}_n   )
$$
is injective. The theorem follows from the fact that $ \mathfrak{v}_n $ and $ \mathbf i \mathfrak{v}_n  $ are isomorphic
as $\pi_1(M^n)$--modules.
\end{proof}

We will use Remark~\ref{remark:Vcomplex} and Equation~(\ref{eqn:su})
to understand the computations for the cohomology with coefficients in $\mathfrak{v}$ in a Riemannian setting.

%

In order to understand the Killing form on $\mathfrak{su}(3,1)$ we follow the exposition of Goldman \cite[4.1.3]{Goldman}. 
Let 
$$
v_+=\begin{pmatrix}
    0 \\ 0 \\ 1 \\ 1
   \end{pmatrix}
\quad
\textrm{ and }
\quad
v_-=\begin{pmatrix}
    0 \\ 0 \\ -1 \\ 1
   \end{pmatrix}
$$ be two null vectors in 
$\mathbf{C}^{3,1}$ representing two distinct boundary points of $\mathbf{H}^3_\mathbf{C}$. Then the element
\[\eta := -\frac 1 2 (v_+ v_-^* - v_- v^*_+) = 
\begin{pmatrix} 
0 & 0 & 0 & 0\\
0 & 0 & 0 & 0\\
0 & 0 & 0 & 1\\
0 & 0 & 1 & 0\end{pmatrix}\]
is the infinitesimal generator of a 1-parameter subgroup of isometries fixing the points $[v_\pm]\in\partial \mathbf{H}^3_\mathbf{C}$ and translating along the geodesic between $[v_+]$ and $[v_-]$.

Decompose the Lie algebra $\mathfrak{su}(3,1)$ into eigenspaces 
\[ \mathfrak g_k = \Ker(\ad_\eta -k \mathbf{I})\]
of $\ad_\eta$. The eigenspace $\mathfrak g_k$ is nonzero only for $k\in\{0,\pm1,\pm2\}$. More explicitly we have:
\begin{equation}
\label{eqn:g0}
\mathfrak g_0 = \Big\{
\begin{pmatrix} 
a & 0 & 0 \\
0 & -\frac{\tr(a)}2 & t\\
0 &  t &-\frac{\tr(a)}2 
\end{pmatrix} \Big| a\in\mathfrak u(2),\; t\in\mathbf R\Big\},
\end{equation}
$\mathfrak g_{\pm1} = \{ v v^*_\pm - v_\pm v^* \mid v\in V(v_+,v_-)^\bot\}$ and
$\mathfrak g_{\pm2} = \{ \mathbf{i} s v_\pm v^*_\pm  \mid s\in \mathbf R\}$
where $V(v_+,v_-)$ denotes the vector space generated by $v_+$ and $v_-$. Note that $V(v_+,v_-)$ is the positive two-dimensional complex subspace of $\mathbf C^{3,1}$ given by $z_3=z_4=0$.
As usual we have $[\mathfrak g_k,\mathfrak g_l]\subset\mathfrak g_{k+l}$ with the convention that $\mathfrak g_{k+l} =0$ if $|k+l|>2$. This tells us immediately that $\mathfrak g_{k}$ is orthogonal with respect to the Killing form to $\mathfrak g_{l}$ for all $k\neq-l$.

Now let $G_\pm\subset PSU(3,1)$ denote the stabilizer of the point  
$[v_\pm]\in\partial \mathbf{H}^3_\mathbf{C}$. The Lie algebra 
$\mathfrak g_\pm$ of $G_\pm$ is given by
\[ \mathfrak g_\pm = \mathfrak g_0 \oplus \mathfrak g_{\pm1}
\oplus \mathfrak g_{\pm2}.\]
Note also that $\mathfrak h_\pm =\mathfrak g_{\pm1}
\oplus \mathfrak g_{\pm2}$ is the Lie algebra of parabolic transformations fixing the point $[v_\pm]$. 

As a consequence of this discussion we obtain the following lemma.

\begin{lemma}
\label{lem:orthogonalparabolic}
The Killing form of $\mathfrak{su}(3,1)$ restricted to 
$\mathfrak g_\pm$   is degenerated. More precisely,
the radical $rad(\mathfrak g_\pm) = \mathfrak g_\pm \cap \mathfrak g_\pm^\bot = \mathfrak h_\pm$ 
consist exactly the infinitesimal parabolic transformations. 
\end{lemma}

\begin{proof}
Let us consider the sign $+$, the other case is analogous.
We have 
\[ \mathfrak g_0^\bot = \mathfrak h_+ \oplus \mathfrak h_-,\quad
\mathfrak g_1^\bot = \mathfrak g_0 \oplus \mathfrak h_+ \oplus \mathfrak g_{-2} \text{ and }
\mathfrak g_2^\bot = \mathfrak g_0 \oplus \mathfrak h_+ \oplus \mathfrak g_{-1}\,.\]
This follows since $\mathfrak g_{k}$ is orthogonal with respect to the Killing form to $\mathfrak g_{l}$ for all $k\neq-l$.
Hence $\mathfrak g_+ \cap \mathfrak g_+^\bot = 
\mathfrak g_+\cap \mathfrak g_0^\bot \cap \mathfrak g_1^\bot
\cap \mathfrak g_2^\bot =\mathfrak h_+=
\mathfrak g_{1}\oplus \mathfrak g_{2}$.
\end{proof}

\section{Invariant subspaces in the complex hyperbolic geometry}
\label{sec:invariantsubspaces}

In this section we shall compute subspaces of the $SO(3,1)$--module $\mathfrak v$ that are invariant by certain elements of
 $PSO(3,1)$. This will be used later for computing certain cohomology groups.
For a given  set of hyperbolic isometries $\Gamma\subset PSO(3,1)$, we let  $\mathfrak{v}^{\Gamma}$ denote the invariant subspace in $\mathfrak{v}$:
$$
\mathfrak{v}^{\Gamma}=\{v\in \mathfrak{v}\mid \Ad_\gamma(v)=v, \ \forall\gamma\in\Gamma\}. 
$$

For our computations, we will view 
elements in $\mathfrak v$ as lying in $\mathbf i \mathfrak v$, namely as infinitesimal isometries of
$\mathbf H^3_\mathbf C$. We shall make use of the  decomposition 
$$ \mathfrak{su}(3,1)^\Gamma = \mathfrak{so}(3,1)^\Gamma \oplus \mathbf i\mathfrak{v}^\Gamma\,.$$
and the following lemma (see  \cite[III.9.3]{Bourbaki} for a proof).
\begin{lemma}
\label{lem:tangentcommutator}
For $\gamma\in PSU(3,1)$,  $\mathfrak{su}(3,1)^\gamma=\Ker(\Ad_\gamma- \1)$ is the Lie algebra of the centralizer of $\gamma$ 
(i.e.\ the Lie subgroup of elements in $PSU(3,1)$ that commute with $\gamma$).
\end{lemma}

Alternatively, the computation of invariant subspaces could also be made with the analogue of Lemma~\ref{lem:tangentcommutator} for
$GL(4)$ or just by explicit computation of the adjoint action on $\mathfrak{v}$.

The centralizer of an element is obtained by means of the stabilizer of an invariant object in $\mathbf H^3_{\mathbf C}\cup \partial \mathbf H^3_{\mathbf C}$.
This explains the organization of this section, one subsection for each object.

\subsection{Geodesics.}
Consider the Riemannian geodesic $\gamma$ in $\mathbf H^3_{\mathbf C}$ between $[v_+]$ and $[v_-]$.
Let $\mathfrak g_0\subset \mathfrak{su}(3,1)$ denote the Lie algebra of the subgroup $G_0\subset PSU(3,1)$ which fixes the endpoints of the geodesic $\gamma$ (see \cite[4.1.3]{Goldman}).
Notice that $G_0\cong \mathbf R\times U(2)$, where $\mathbf R$ acts by translations and $ U(2)$ is the pointwise stabilizer, isomorphic to the stabilizer of a point in 
$\mathbf H^2_{\mathbf C}$, hence  $\mathfrak g_0\cong\mathbf R\oplus \mathfrak{u}(2)$.

\begin{lemma} 
\label{lem:invhyperbolic}
Let $A\in PSO(3,1)$ be a hyperbolic element of complex length $l+\mathbf i\,  \alpha$, $l\neq 0$.
\begin{enumerate}[(i)]
\item  If $\alpha\not\in\pi\mathbf Z$, then $\dim \mathfrak{v}^{A}=1$.
\item  If $\alpha\in\pi\mathbf Z$, then $\dim \mathfrak{v}^{A}=3$.
\end{enumerate}
\end{lemma}

\begin{proof}
We let $\gamma$ denote  the axis of $A$. 
After conjugation we might assume that $\gamma$ is the geodesic between $[v_+]$ and $[v_-]$ and hence
\[ A=
\begin{pmatrix} 
\cos\alpha & -\sin\alpha &0 & 0\\ 
\sin\alpha & \cos\alpha& 0&0\\
0 & 0 &\cosh l &\sinh l\\
0 & 0 &\sinh l &\cosh l\\
\end{pmatrix}.\]

If $\alpha\in \pi\mathbf Z$, then $A$  commutes with the whole stabilizer 
$G_0$, with Lie algebra $\mathfrak{g}_0$ (see Equation~(\ref{eqn:g0})).
The elements of $\mathbf i\,\mathfrak{v}^A=\mathbf i\, \mathfrak{v}\cap \mathfrak{g}_0$ are of the form 
$$
\begin{pmatrix} 
b \, \mathbf i & a \, \mathbf i &0 & 0\\ 
a \, \mathbf i & c \, \mathbf i & 0&0\\
0 & 0 & -\frac{b+c}2  \mathbf i & 0\\
0 & 0 & 0 & -\frac{b+c}2 \mathbf i\\
\end{pmatrix}, \textrm{ with }a, b, c\in\mathbf R.
$$
 Hence $\dim \mathfrak{v}^A =3$. 

If $\alpha\not \in \pi\mathbf Z$, then the elements of  $\mathbf i\, \mathfrak{v}^A$ are as before, but by setting $a=0$ and $b=c$, hence $\dim \mathfrak{v}^A=1$.
This corresponds to the $\mathfrak{u}(1)$ factor in 
 $\mathfrak g_0\cong\mathbf R\oplus \mathfrak{u}(2)$.
\end{proof}
\begin{remark}\label{rem:lox}
Note that if $A\in PSO(3,1)$ is a loxodromic element with complex length $l+\mathbf i \alpha$ with $l\neq0$ and $\alpha\not\in\pi\mathbf Z$ then $\mathfrak{v}^A=\mathfrak{v}^{G_0}$ is one-dimensional generated by
the vector 
$$
\begin{pmatrix} 
-1 & 0 &0 & 0\\ 
0 & -1 & 0&0\\
0 & 0 & 1 & 0\\
0 & 0 & 0 & 1
\end{pmatrix}\,.
$$
\end{remark}

\subsection{Complex hyperbolic lines}
The complex hyperbolic space is the projectivization of the subset of the time-like vectors of $\mathbf C^{3,1}$. 
 A \emph{complex hyperbolic line} is defined as the intersection of 
 $\mathbf H^3_{\mathbf C}$ with a complex projective line.
The group $SU(3,1)$ acts transitively on the  set of complex planes that contain time-like vectors. Hence all complex hyperbolic lines are  
isomorphic to $\mathbf H^1_\mathbf{C}$, and a standard model for a complex hyperbolic line is the image of the plane given by $x_1=x_2=0$. 
The intersection of a complex hyperbolic line with 
$\partial \mathbf H^3_\mathbf{C}$ is a smooth circle called a \emph{chain}. Two distinct boundary points  of $\mathbf H^3_\mathbf{C}$ are contained in a unique chain and the Riemannian geodesic between the two boundary points is contained in the corresponding complex hyperbolic line.

The identity component of the stabilizer of a chain is given by
$P(U(2)\times U(1,1)) \subset PSU(3,1)$.

\begin{lemma} Let $A\in PSO(3,1)$ be an elliptic element of rotation angle $\alpha\in (0,2\pi)$.
\begin{enumerate}[(i)]
\item If $\alpha=\pi$, then $\dim \mathfrak{v}^{A}=5$.
\item If $\alpha\neq\pi$, then $\dim \mathfrak{v}^{A}=3$.
\end{enumerate}
\end{lemma}

\begin{proof} As before we let $\gamma$ denote  the axis of $A$. 
After conjugation we might assume that $\gamma$ is the geodesic between $[v_+]$ and $[v_-]$ and hence
\[ A=
\begin{pmatrix} 
\cos\alpha & -\sin\alpha &0 & 0\\ 
\sin\alpha & \cos\alpha& 0&0\\
0 & 0 &1 &0\\
0 & 0 &0 &1
\end{pmatrix}.\]

If $\alpha=\pi$ then a direct calculation gives that $\mathbf i\, \mathfrak{v}^A$ consists of matrices of the form:
$$
\begin{pmatrix} 
a \, \mathbf i & b \, \mathbf i  &0 & 0\\ 
b\, \mathbf i  & c \, \mathbf i & 0&0\\
0 & 0 & d \, \mathbf i &   e \, \mathbf i \\
0 & 0 & - e \, \mathbf i & - (a+c+d)\, \mathbf i\\
\end{pmatrix}, \quad \text{with $a, b, c,d,e \in\mathbf R$.}
$$
Hence $\dim \mathfrak{v}^A=5$. Notice that  $\mathbf i\, \mathfrak{v}^A$ is the imaginary part of  $\mathfrak{s}(\mathfrak{u}(2)\oplus \mathfrak{u}(1,1))$.

If  $\alpha\neq \pi$, then  $\mathbf i\, \mathfrak{v}^A$ consists of the previous matrices that in addition satisfy $b=0$ and $a=c$. Hence
$\mathbf i\, \mathfrak{v}^A$ is the imaginary part of $\mathfrak{u}(1,1)$ (viewed in $\mathfrak{s}(\mathfrak{u}(2)\oplus \mathfrak{u}(1,1))$) and has dimension $3$.
\end{proof}

\subsection{Points at infinity and the Heisenberg geometry}
\label{subsection:Heisenberg}

In the sequel we will use the notation of Section~\ref{subsectioncomplex},
i.e.\ we will fix two light-like vectors $v_\pm\in\mathbf C^{3,1}$ representing 
two distinct boundary points $[v_\pm]\in\partial \mathbf H^3_\mathbf{C}$. Moreover we will use the root-space decomposition of $\mathfrak{su}(3,1)$. The Heisenberg group $\mathcal H_-$ is the group of parabolic transformations fixing the point $[v_-]$, i.e.\
$\exp\co\mathfrak g_{-1}\oplus \mathfrak g_{-2}\to\mathcal H_-$ is given by 
\begin{align}
\exp(v_- v^* - v v_-^*&+ \mathbf{i}\ t\ v_-v_-^*)\notag \\ &=  I_4 + v_- v^* - v v_-^*-
(\|v\|^2/2 -\mathbf{i}t) v_-v_-^*\notag\\
&=
\begin{pmatrix} 
1 & 0 & z_1 & z_1 \\
0 & 1 & z_2 & z_2 \\
-\bar z_1 & -\bar z_2 & 1-\|v\|^2/2 +\mathbf{i}t & -\|v\|^2/2 +\mathbf{i}t \\
\bar z_1 & \bar z_2 & \|v\|^2/2 -\mathbf{i}t & 1+\|v\|^2/2 -\mathbf{i}t
\end{pmatrix}\notag \\ &=: H(z_1,z_2,t) \label{eqn:heisemberg}
\end{align}
where $v=(z_1,z_2,0,0)^t\in v_+^\bot\cap v_-^\bot$ is a space-like vector and hence $\langle v,v\rangle = \|v\|^2= |z_1|^2+|z_2|^2 \geq 0$.

Following the exposition in Goldman's book \cite[4.2]{Goldman}, the boundary at $\infty$ of $\mathbf H_{\mathbf C}^3$  minus the point $[v_-]$ 
can be identified with a \emph{Heisenberg space}, i.e.\ a space equipped with a simply transitive left action of the Heisenberg group 
$\mathcal H_-$. Hence by looking at the orbit of $[v_+]$ we have a bijection
$\mathcal H_-  \to \partial \mathbf H_{\mathbf C}^3\setminus\{[v_-]\}$ given by
\[  H(z_1,z_2,t) \mapsto H(z_1,z_2,t) [v_+] = 
\begin{bmatrix}  2z_1\\ 2z_2\\ 
1-\|z\|^2 +2\mathbf i t\\ 
1+\|z\|^2   -2\mathbf i t\end{bmatrix}\]
where $\|z\|^2= |z_1|^2 + |z_1|^2$.

In the sequel we shall  represent points of $\mathcal H_-$ by triples of points $(z_1,z_2,t)$ where $z_1,z_2\in \mathbf C$, $t\in\mathbf R$ with multiplication
\begin{multline}
(\omega_1,\omega_2, s)\cdot (z_1,z_2,t) = (\omega_1+z_1,\omega_2+z_2,s+t+\operatorname{Im}(\omega_1\bar z_1+\omega_2\bar z_2) ),
\\
\forall  (\omega_1,\omega_2, s), (z_1,z_2,t)\in\mathcal H.
\end{multline}
Therefore, $\mathcal H_-$ is a nilpotent $5$--dimensional real  Lie group, which is a nontrivial central extension
\[ 0\to\mathbf R\to \mathcal H_-\to\mathbf C^2\to 0\,.\]
The center are the elements of the form $(0,0,t)$, $t\in\mathbf R$.

In the sequel we will make use of the \emph{Siegel domain model} 
$\mathfrak H^3$ of $\mathbf H_\mathbf{C}^3$. Here
\[ \mathfrak H^3 =\Big\{ w =\begin{pmatrix} w_1\\w_2\\w_3\end{pmatrix}\in\mathbf C^3 \; \Big| \;|w_1|^2 + |w_2|^2 < 2\Re(w_3)\Big\}\]
is obtained in the following way: we choose the point $[v_-]\in\partial\mathbf{H}_\mathbf{C}^3$ and we denote by $H\subset \mathbf P^{3,1}$ the projective hyperplane tangent to $\partial\mathbf{H}_\mathbf{C}^3$ at $[v_-]$. More precisely, $H$ is the projectivization of $v_-^\bot \subset \mathbf C^{3,1}$ given by the equation $z_3+z_4 =0$. The corresponding affine embedding $\mathbf C^3 \to \mathbf{C P}^{3}\setminus H$ is given by
\[  \begin{pmatrix} w_1\\w_2\\w_3\end{pmatrix} \mapsto
\begin{bmatrix}  w_1\\  w_2\\ 1/2 - w_3\\ 1/2 + w_3\end{bmatrix}\;.
\]
It is easy to see that $\mathbf{H}_\mathbf{C}^3$ corresponds to the Siegel domain $\mathfrak H^3\subset \mathbf C^3$.
In this model the whole stabilizer $G_-$ of the point $[v_-]$ at infinity is the semidirect product:
$$
G_-=\mathcal H_- \rtimes (U(2)\times \mathbf R)\,.
$$
Here $U(2)$ acts linearly on the factor $\mathbf C^2$, and trivially on the factor $\mathbf R$. Moreover $\mathbf R$ acts as follows:
$$
(I_2,\lambda) (z_1,z_2,t) (I_2,-\lambda)=(e^{-\lambda} z_1,e^{-\lambda} z_2, e^{-2\lambda} t), \ \forall \lambda\in\mathbf R,\ \forall (z_1,z_2,t)\in\mathcal H.
$$
Hence the product on $G_{-}$ is given by :
\begin{align*}
(z_1,&z_2,t) (A,\lambda) \cdot (z'_1,z'_2,t') (A',\lambda') \\
&= (z_1,z_2,t) \big(e^{-\lambda}(z_1',z_2')A^{t} ,e^{-2\lambda}t'\big) (AA',\lambda+\lambda'), 
\end{align*}
for all $(z_1,z_2,t),(z_1',z_2',t') \in\mathcal H$, $A,A'\in U(2)$ and $\lambda,\lambda'\in \mathbf R$.

In this construction, the subgroup of real parabolic transformations corresponds to $\mathbf R^2\times \{0\}\subset \mathcal H_-$.


\begin{lemma}
\label{lem:invparabolic}
\begin{enumerate}[(i)]
\item If $A$ is a nontrivial parabolic element of $PSO(3,1)$, then  
$
\dim \mathfrak{v}^A=3
$.
\item
If $\Gamma < PSO(3,1)$ is a rank 2 parabolic subgroup, then
$
\dim \mathfrak{v}^\Gamma=1
$.
\end{enumerate}
\end{lemma}

\begin{proof}
Using the representation in the Heisenberg group $\mathcal H_-$, we may assume that up to conjugation $A$ is $(1,0,0)\in  \mathcal H_-$.
Note that the centralizer of $A$ is contained in $G_-$. This follows from the fact that $A$ has a unique fixed point on 
$\overline{\mathbf H^3_\mathbf{C}}$ and every element which commutes with $A$ has to fix this point.

Now a direct calculation gives that the centralizer of $A$ in $G_-$ is $5$--dimensional and given by
\[ \Big\{ ( s, z, t) \begin{pmatrix} 1 & 0 \\ 0 & a\end{pmatrix} \in G_- \mid s,t\in \mathbf R,\, z\in\mathbf C\text { and } a\in U(1)  \Big\}.\] 

Thus $\dim (\mathfrak{su}(3,1)^A)=5$, and since $\dim (\mathfrak{so}(3,1))^A=2$ (the tangent space to the real parabolic group itself), the first assertion follows.

For the last assertion, we view $\Gamma$ as a rank 2 subgroup of the Heisenberg group
$$
\Gamma < \mathbf R^2\times\{ 0\} < \mathcal H_-.
$$ 
Its centralizer is contained in $G_-$ and is precisely the subgroup of elements with real coordinates:
$$
\mathbf R^3\cong \{(s_1,s_2,t)  \in \mathcal H_- \mid s_1,s_2,t\in \mathbf R\} < \mathcal H_-\,.
$$
As the subgroup of real parabolic transformations
$\mathbf R^2\times \{0\}$ is  the centralizer   of $\Gamma$  in $PSO(3,1)$, it follows that 
$\mathfrak{v}^\Gamma=\{(0,0)\}\times\mathbf R $ is one dimensional.
\end{proof}

\section{The cohomology of the torus}
\label{sec:cohmologytorus}

In this section, we analyze the cohomology of the boundary $\partial M$ and the image
of the map induced by inclusion $\partial M\subset M$, which is a Lagrangian subspace.

\subsection{A Lagrangian subspace}

As in Section~\ref{sec:Thurstonsslice}, let $\rho_u$ denote a representation contained in Thurston's slice, where $u=(u_1,\ldots,u_k)\in U\subset \mathbf C^k$
is a point in the deformation space.
The subspace invariant by the image of the peripheral subgroup of the $i$th component is denoted by $\mathfrak{v}^{\rho_u(\pi_1(\partial_i M))}$,
and its orthogonal complement by
$$
\big(\mathfrak{v}^{\rho_u(\pi_1(\partial_i M))}\big)^\bot=
\{ v\in \mathfrak{v} \mid B(v,w)=0,\ \forall w\in \mathfrak{v}^{\rho_u(\pi_1(\partial_i M))}\}. 
$$

\begin{lemma} \label{lem:invkilling}
\begin{enumerate}[(i)]
	\item For $u_i\neq 0$, the radical of 
	$\mathfrak{v}^{\rho_u(\pi_1(\partial_i M))}$ is trivial, i.e.\
$$\big(\mathfrak{v}^{\rho_u(\pi_1(\partial_i M))}\big)^\bot\cap \mathfrak{v}^{\rho_u(\pi_1(\partial_i M))}= 0.$$ 

              \item For every $u\in U$, the invariant subspace $\mathfrak{v}^{\rho_u(\pi_1(\partial_i M))}$ has dimension one.
\end{enumerate}
\end{lemma}

\begin{proof}
When $u_i\neq 0$, $\rho_u(\pi_1(\partial_i M))$ consists of non-parabolic isometries that preserve a geodesic,
and we want to apply Lemma~\ref{lem:invhyperbolic} (i) and Remark~\ref{rem:lox}. For this, we need to find an element $\gamma\in\pi_1(\partial_i M)$ such that $\rho_u(\gamma )$ satisfies the hypothesis of 
Lemma~\ref{lem:invhyperbolic} (i), namely that $\rho_u(\gamma )$ has nonzero translational part and its rotation angle is not an integer multiple of $\pi$.
If the real part of 
$u_i$ does not vanish and the imaginary part of $u_i$ 
is not contained in $\mathbf Z\pi$  then we choose $\mu_i $.
If the real part of $u_i$ vanishes, then by Theorem~\ref{thm:thurston} the real part of $v_i$ does not, and the condition on the complex length applies to either $\gamma=\lambda_i$
or $\gamma=\lambda_i\mu_i$, that have respective complex lengths $v_i$ and $u_i+v_i$. The same argument applies when the imaginary part of $u_i$ is zero. 

By Lemma~\ref{lem:invhyperbolic} (i) and its proof,  $\mathfrak{v}^{\rho_u(\pi_1(\partial_i M))}$ is the one dimensional subspace generated by (a conjugate of) 
$$
\left( \begin{smallmatrix}
    	1&   &  & \\
	 & 1 &   &    	\\
	 &  & -1 &   \\
	 &  &   & -1 \\
\end{smallmatrix}\right), 
$$ 
which is a non-isotropic element for the Killing form,
and both assertions of the lemma are clear when $u_i\neq 0$. 

When $u_i=0$, assertion (ii) is precisely Lemma~\ref{lem:invparabolic}~(ii).
\end{proof}

\begin{cor}\label{cor:h1boundary}
For every $u\in U$ we have $H^j(\partial M;\mathfrak{v}_{\Ad\rho_u})=0$ for $j>2$ and
\begin{align*}
\dim H^0(\partial M;\mathfrak{v}_{\Ad\rho_u}) &= k, \\
\dim H^1(\partial M;\mathfrak{v}_{\Ad\rho_u}) &=2k,\\
\dim H^2(\partial M;\mathfrak{v}_{\Ad\rho_u}) &=k\,.
\end{align*}
\end{cor}
\begin{proof}
We have 
$$H^0(\partial M;\mathfrak{v}_{\Ad\rho_u})\cong 
\bigoplus_{i=1}^k\mathfrak{v}^{\rho_u(\pi_1(\partial_i M))}$$
and hence by Lemma~\ref{lem:invkilling}~(ii) we obtain 
$$\dim H^0(\partial M;\mathfrak{v}_{\Ad\rho_u}) = k\,.$$ 
Now Poincar\'e duality gives
$\dim H^2(\partial M;\mathfrak{v}_{\Ad\rho_u}) = k$ and since the Euler characteristic of 
$\partial M$ vanishes we obtain $\dim H^1(\partial M;\mathfrak{v}_{\Ad\rho_u}) = 2k$.
\end{proof}

The cup product on $H^1(\partial M;\mathfrak{v}_{\Ad\rho_u})$ is the orthogonal sum of the cup products on the groups 
$H^1(\partial_i M;\mathfrak{v}_{\Ad\rho_u})$. More precisely, if we denote by
$\mathit{res}_i\co H^1(\partial M;\mathfrak{v}_{\Ad\rho_u})\to H^1(\partial_i M;\mathfrak{v}_{\Ad\rho_u})$ the restriction induced by the inclusion
$\partial_i M\hookrightarrow \partial M$, then for 
$z_1,z_2\in H^1(\partial M;\mathfrak{v}_{\Ad\rho_u})$ we have
\begin{equation}
\label{eqn:cupsum} 
z_1\cup z_2 = \sum_{i=1}^k \mathit{res}_i(z_1)\cup\mathit{res}_i(z_2)\,.
\end{equation}
Note that this defines a symplectic form 
$$\omega\co H^1(\partial M;\mathfrak{v}_{\Ad\rho_u})\otimes H^1(\partial M;\mathfrak{v}_{\Ad\rho_u})\to\mathbf R$$
given by $\omega(z_1,z_2)=z_1\cup z_2$.

\begin{lemma} 
\label{lem:rankone} 
Let $u=(u_1,\ldots,u_k)\in U$.
\begin{enumerate}[(i)]
              \item When $u_i\neq 0$, there is a natural isomorphism
$$
H^*(\partial_i M; \mathfrak{v}_{\Ad\rho_{u}})\cong H^*(\partial_i M;\mathbf R)\otimes \mathfrak{v}^{\rho_u(\pi_1(\partial_i M))}.
$$

              \item For  $u\in U$, $\dim H^1(\partial M; \mathfrak{v}_{\Ad\rho_{u}})=2\, k$, and the image of the map
$$H^1(M; \mathfrak{v}_{\Ad\rho_{u}})\to H^1(\partial M; \mathfrak{v}_{\Ad\rho_{u}})$$ is a Lagrangian subspace 
of $H^1(\partial M; \mathfrak{v}_{\Ad\rho_u})$ for the form $\omega$
(in particular it has dimension $k$).
\end{enumerate}
\end{lemma}

\begin{proof}
To prove assertion (i), we use the decomposition of Lemma~\ref{lem:invkilling}:
$$
\mathfrak{v}= \big(\mathfrak{v}^{\rho_u(\pi_1(\partial_i M))}\big)^\bot\oplus \mathfrak{v}^{\rho_u(\pi_1(\partial_i M))},
$$
which is a direct sum of $\pi_1(\partial_i M) $--modules, and therefore it induces
a direct sum in cohomology. Since $\big(\mathfrak{v}^{\rho_u(\pi_1(\partial_i M))}\big)^\bot$ has no invariant subspaces, 
$$
H^0(\partial_i M, \big(\mathfrak{v}^{\rho_u(\pi_1(\partial_i M))}\big)^\bot)=0.
$$ 
In addition, the Killing form restricted 
to $     \big(\mathfrak{v}^{\rho_u(\pi_1(\partial_i M))}\big)^\bot     $ is non-degenerate, thus by duality and by vanishing of the Euler characteristic
$$
H^*(\partial_i M, \big(\mathfrak{v}^{\rho_u(\pi_1(\partial_i M))}\big)^\bot)=0.
$$
 Hence
$$
  H^*(\partial_i M; \mathfrak{v})= H^*(\partial_i M;  \mathfrak{v}^{\rho_u(\pi_1(\partial_i M))}   )\cong H^*(\partial_i M;\mathbf R)\otimes \mathfrak{v}^{\rho_u(\pi_1(\partial_i M))}. 
$$

The first statement of assertion~(ii) is just Corollary~\ref{cor:h1boundary}.
The fact that the image of the map $H^1(M; \mathfrak{v}_{\Ad\rho_{u}})\to H^1(\partial M; \mathfrak{v}_{\Ad\rho_{u}})$ is a Lagrangian subspace follows from duality.
We reproduce the proof for completeness (cf.\ \cite{HodgsonThesis}).
We are interested in the following part of the exact cohomology sequence of the pair $(M,\partial M)$:
\[ 
H^1(M;\mathfrak{v}_{\Ad\rho_{u}})  
\stackrel{j^*}{\longrightarrow}  
H^1(\partial M; \mathfrak{v}_{\Ad\rho_{u}}) \stackrel{\Delta}{\longrightarrow} 
H^2(M, \partial M; \mathfrak{v_{\Ad\rho_{u}}})\,. \]
The maps $j^*$ and $\Delta$ are dual to each other: 
for $z_1\in H^1(M;\mathfrak{v}_{\Ad\rho_{u}})$ and 
$z_2\in H^1(\partial M; \mathfrak{v}_{\Ad\rho_{u}})$,
\[
\langle j^*(z_1)\cup z_2, [\partial M]\rangle = 
\langle z_1 \cup \Delta (z_2),[M,\partial M] \rangle,
\]
where $[M,\partial M]\in H_3(M,\partial M;\mathbf R)$ and $[\partial M]\in H_2(\partial M;\mathbf R)$
denote the respective fundamental classes.

It follows that
$\dim \operatorname{Im}( j^*) = \frac12 \dim H^1(\partial M; \mathfrak{v}_{\Ad\rho_{u}}) =k$.
 Moreover 
$\Delta\circ j^* =0$ implies that  $\operatorname{Im}( j^*)$ is isotropic and hence Lagrangian since $\dim \operatorname{Im}( j^*) =k$.
\end{proof}

\begin{cor}
\label{cor:dim}
Let $M$ be a cusped manifold. 
Then for all $u\in U\subset \mathbf C^k$ we have 
$$\dim H^1(M; \mathfrak{v}_{\Ad\rho_{u}})\geq k\,. $$
Moreover, $M$ is infinitesimally projectively rigid iff $\dim H^1(M; \mathfrak{v}_{\Ad\rho_{0}})=  k$.
\end{cor}
\begin{proof} The proof follows directly from Lemma~\ref{lem:rankone} and from the decomposition of the $SO(3,1)$--module 
$\mathfrak{sl}(4)= \mathfrak{so}(3,1)\oplus \mathfrak{v}$
(see Equation~(\ref{eqn:sl})). \end{proof}

\subsection{Parabolic representations}

Let ${\lambda}$ and ${\mu}$ be two generators of $\mathbf Z^2$ and 
$$
\varrho\co\mathbf Z^2\to PSO(3,1)
$$ 
a representation into a parabolic group.
Up to conjugation we suppose that the boundary point $[v_-]$ is the fixed point of the parabolic group. 
Viewing the parabolic group as translations of $\mathbf R^2$,
$\varrho({\lambda})$ is a translation of vector $v_ {\lambda}$, and $\varrho({\mu})$
of vector $v_ {\mu}$. Assume that the representation has rank $2$, (i.e.\
$v_ {\lambda}$ and $v_ {\mu}$ are linearly independent).
Then:

\begin{lemma} 
\label{lem:angles}
If the angle $\varphi$ between $v_ {\lambda}$ and $v_ {\mu}$ is not in 
$\frac\pi3\mathbf Z$ then
the map induced by restrictions
$$
H^1(\mathbf Z^2;\mathfrak{v}_{\Ad\varrho}) \xrightarrow{i_{\lambda}^*\oplus i_{\mu}^*} H^1({\lambda};\mathfrak{v}_{\Ad\varrho})\oplus H^1({\mu};\mathfrak{v}_{\Ad\varrho})
$$
is injective. Moreover, $\operatorname{rank}(i_{\lambda}^*)=\operatorname{rank}(i_{\mu}^*)=1$.
\end{lemma}

\begin{proof}
We follow the notation from Subsection~\ref{subsection:Heisenberg}.
We may assume that $v_ {\lambda}=(1,0)$, 
$v_ {\mu}=(a \cos\varphi, a\sin\varphi )\in\mathbf R^2$, $a\,\sin\varphi\neq 0$. 
In the Heisenberg model $\mathcal H_-$, 
$\varrho({\lambda})=(1,0,0)$ and $\varrho({\mu})=(a \cos\varphi, a\sin\varphi ,0)$.
For $\theta\in\mathbf R$, we define a representation $\varrho_{\theta}\co \mathbf Z\oplus\mathbf Z\to G_-$ by
$$
\varrho_{\theta}({\lambda})= \varrho({\lambda})
\quad\textrm{ and }\quad
\varrho_{\theta}({\mu})=
\begin{pmatrix} 1& 0 \\ 0 & e^{\mathbf i \theta}\end{pmatrix} 
 \varrho({\mu}) .  
$$
Notice that $\varrho_{\theta}({\lambda})$ and $  \varrho_{\theta}({\mu})$ commute, because 
 $\left(\begin{smallmatrix} 1& 0 \\ 0 & e^{\mathbf i \theta}\end{smallmatrix}\right)\in U(2)$ 
 fixes $(1,0)$.

Differentiating at $\theta=0$, we obtain an infinitesimal deformation i.e.\ a cocycle 
$d_\mu\co\mathbf Z^2\to \mathfrak g_- = \mathfrak g_0\oplus \mathfrak g_{-1}\oplus \mathfrak g_{-2}$ given by
\[ d_\mu(\gamma) = \frac{ d \varrho_{\theta}(\gamma) }{d\theta}\bigg|_{\theta=0} \varrho_{0}(\gamma)^{-1}\,.\]
The cocycle $d_\mu\co \mathbf Z^2 \to \mathfrak{g}_-$  is trivial when restricted to ${\lambda}$. More precisely we obtain
\[ d_\mu(\lambda) = 0 \quad \text{ and }\quad
d_\mu(\mu) = \begin{pmatrix} 0& 0 \\ 0 & \mathbf i \end{pmatrix} .  
\]


Notice that the derivative of the canonical embedding
$U(2)\to PSU(3,1)$ determined by
\[ A \mapsto \begin{pmatrix} A& 0  \\ 0 & I_2  \end{pmatrix} \]
is the map $\mathfrak u(2)\to \mathfrak {su}(3,1)$ given by
\[ a \mapsto \begin{pmatrix} a& 0  \\ 0 &  0 \end{pmatrix} - \frac{\tr a} 4 I_4\,\]
and that 
\[ \begin{pmatrix} 0& 0 \\ 0 & \mathbf i \end{pmatrix} \mapsto \frac{\mathbf i }4
\begin{pmatrix}
- 1 & & & \\ 
& 3 & & \\
& & -1 & \\
& & & -1
\end{pmatrix} \in \mathbf i \mathfrak v\,.
\]

Hence we obtain a  cocycle 
$z_\mu\co \mathbf Z^2  \to  \mathfrak{v}$ given by
$z_\mu(\lambda)= 0$ and  $z_\mu(\mu) = a_{\lambda}$ where
\[
a_\lambda := 
\begin{pmatrix}
- 1 & & & \\ 
& 3 & & \\
& & -1 & \\
& & & -1
\end{pmatrix} \in  \mathfrak v\,.
\]

In the same way we obtain a second cocycle 
$z_\lambda\co \mathbf Z^2  \to  \mathfrak{v}$ given by
$z_\lambda(\lambda)= a_\mu$ and  $z_\lambda(\mu) = 0$
where
\begin{equation*}
a_ {\mu}=\label{eqn:infbending2} 
\begin{pmatrix}
1-2\cos(2\varphi) & -2\sin(2\varphi) &  & & \\ 
-2\sin(2\varphi)& 1+2\cos(2\varphi)&  & & \\
& & -1 & \\
& & & -1
\end{pmatrix}\in\mathfrak v\,.
\end{equation*}
Here $\varphi$ is the angle between 
$v_{\mu}$ and $v_{\lambda}$.
Notice that $z_\lambda$ is constructed as $z_\mu$ but switching the roles of $\lambda$ and $\mu$.
Thus $a_{\mu}$  is invariant by $\varrho({\mu})$, and
it can be obtained by conjugating $a_{\lambda}$ by a rotation of angle $\varphi$. 

We claim that the cocycle $z_\mu$
is cohomologically nontrivial when restricted to $\mu$, i.e.  nontrivial in $H^1({\mu};\mathfrak{v}_{\Ad\varrho})$.
This  proves that $z_\mu$ is a nontrivial cocycle,
 and $\operatorname{rank}(i_{\mu}^*)\geq 1$. 
By symmetry of the generators, $z_\lambda$ is a nontrivial cocycle
and 
$\operatorname{rank}(i_{\lambda}^*)\geq 1$. Moreover, since
$i_{\mu}^*(z_\lambda)=0= i_{\lambda}^*(z_\mu)$ it follows that the image of $i_{\mu}^*\oplus i_{\lambda}^*$ is $2$--dimensional and the assertion of the lemma follows.

To prove the claim, we will use the cup product 
$$
H^1({\mu};\mathfrak{v}_{\Ad\varrho})\otimes H^0({\mu};\mathfrak{v}_{\Ad\varrho})\to H^1({\mu};\mathbf R)\cong \mathbf R
$$ 
associated to the Killing form defined in (\ref{eqn:cupproduct}).
Recall that $a_\mu\in H^0({\mu};\mathfrak{v}_{\Ad\varrho})= \mathfrak{v}^{\varrho({\mu})}$ is invariant under the action of $\mu$.
The cup product $i_{\mu}^* (z_\mu)\cup a_ {\mu}$ is a represented by the homomorphism $H_1(\mu; \mathbf R)\to\mathbf R$ given by
\begin{align*}
 \big(i_{\mu}^* (z_\mu)\cup a_ {\mu}\big)(\mu) &= 
 B(a_\lambda,a_ {\mu})= 8\tr (a_\lambda\cdot a_ {\mu})\\&=32(1+2\cos(2\varphi))
 =128\big(\cos^2(\varphi) - \frac 1 4\big)\,.
\end{align*}
This is nonzero by the hypothesis about the angle $\varphi$ between $v_ {\lambda}$ and $v_ {\mu}$,
hence $i_{\mu}^* (z_ \mu)\cup a_ {\mu}$ is not homologous to zero.
\end{proof}

\begin{remark}
\begin{enumerate}
\item Notice that in the proof of Lemma~\ref{lem:angles}, instead of the cup product
we could have considered the Kronecker paring between homology and cohomology,
and we would have ended up checking the non-vanishing of the same evaluation of the Killing 
form $B(a_ {\lambda},a_ {\mu})$.
\item Note that the assumption $\varphi\neq \pi/3$ is essential in 
Lemma~\ref{lem:angles}:
we can still construct the cocycles $z_{\mu}$ and $z_{\lambda}$ in the case $\varphi =\pi/3$. But now we have 
$i_{\mu}^{*}(z_{\mu}) = 0 = i_{\lambda}^{*}(z_{\lambda})$. Moreover the
two cocycles $z_{\mu}$ and $z_{\lambda}$ represent linear dependent nontrivial cohomology classes in $H^{1}(\mathbf Z; \mathfrak{v}_{\Ad\varrho})$.
Hence the map 
\[ 
i_{\lambda}^*\oplus i_{\mu}^*\co H^1(\mathbf Z^2;\mathfrak{v}_{\Ad\varrho}) \to H^1({\lambda};\mathfrak{v}_{\Ad\varrho})\oplus H^1({\mu};\mathfrak{v}_{\Ad\varrho})
\]
is not injective if $\varphi=\pi/3$.
\end{enumerate}
\end{remark}


\smallskip

Before the next lemma, we still need a claim about symplectic forms on vector spaces.

\begin{claim} 
\label{lem:cupproduct}
Let $(V,\omega)$ be a $2$--dimensional symplectic subspace.
Suppose that $f,g\co V\to\mathbf R$ are linear forms which form a basis of the dual space $V^*$, i.e.\ $f\oplus g\co V\to \mathbf R^2$ is an isomorphism.

Then there exists a constant $c\in\mathbf R$, $c\neq 0$, such that, 
for every $x,y\in V$
$$
\omega(x,y)= c (f(x) g(y) - g(x) f(y)) \,.
$$
\end{claim}

\begin{proof}
The claim is a consequence of the fact that the space of antisymmetric bilinear forms on $\mathbf R^2$ is one dimensional.
\end{proof}

\begin{lemma}
\label{lem:almosteveryslope}
If a subspace 
$L\subset H^1(\partial M; \mathfrak{v}_{\Ad \rho_0})$  is Lagrangian for the cup product, then
there exist simple closed curves $\mu_1\in \pi_1(\partial_1 M)$, \ldots,
$\mu_k\in \pi_1(\partial_k M)$ so that the image of $L$ injects in
$H^1(\mu_1;\mathfrak{v}_{\Ad \rho_0})\oplus \cdots \oplus H^1(\mu_k;\mathfrak{v}_{\Ad \rho_0}) $.
Moreover, injectivity fails if we consider only $k-1$ curves.
\end{lemma}

\begin{proof} Along this proof, the action on $\mathfrak v$ is the adjoint of the holonomy of the complete structure, so $\Ad\rho_0$ is omitted from notation.
For $j=1,\ldots,k$, let $res_j\co H^1(\partial M; \mathfrak v)\to H^1(\partial_j M; \mathfrak v)$ denote the map induced by restriction, which is also the projection to the $j$th factor
of the isomorphism
$$
H^1(\partial M; \mathfrak v)\cong H^1(\partial_1 M; \mathfrak v)\perp\cdots\perp H^1(\partial_k M; \mathfrak v).
$$
Recall that this is an orthogonal sum for the cup product (\ref{eqn:cupsum}).

We prove the lemma by induction on $k$. When $k=1$, it suffices to chose two curves
$\mu_1$ and $\lambda_1$ in $\partial_1 M$ that satisfy the hypothesis of Lemma~\ref{lem:angles}.
Hence 
$$ i_{\mu_1}^*\oplus i_{\lambda_1}^*\co  
H^1(\partial_1M;\mathfrak{v})\to H^1(\mu_1;\mathfrak{v}) \oplus H^1(\lambda_1;\mathfrak{v})$$
is injective. 
Then for at least one of the curves, say $\mu_1$, $i_{\mu_1}^*(L)\neq 0$.

For the induction step, we chose the corresponding curves on the $k$--th component $\mu_k$ and $\lambda_k$, so that
$$ i_{\mu_k}^*\oplus i_{\lambda_k}^*\co  
H^1(\partial_kM;\mathfrak{v})\to H^1(\mu_k;\mathfrak{v}) \oplus H^1(\lambda_k;\mathfrak{v})$$
is injective, and assume that 
$i_{\mu_k}^*(L)\neq 0$.

Let $L'\subset H^1(\partial_1M;\mathfrak{v})\perp \cdots\perp H^1(\partial_{k-1} M;\mathfrak{v})$ be the projection 
to the first $k-1$ factors of the kernel of $i_{\mu_k}^*$ restricted to $L$; i.e.
$$
L'= (res_1\oplus\cdots\oplus res_{k-1})(\ker i_{\mu_k}^*\vert_L)
$$
We first check that $L'$ is  isotropic. Given $x,y\in L'$, there exist $x_k,y_k\in H^1(\partial_{k} M;\mathfrak{v})$ such that 
$(x,x_k)$, $(y,y_k)\in L$ and $i_{\mu_k}^*(x_k)=i_{\mu_k}^*(y_k)=0$. Thus, by Claim~\ref{lem:cupproduct} and Equation~(\ref{eqn:cupsum}):
$$
0=(x,x_k)\cup(y,y_k)=x\cup y + c_k (i_{\mu_k}^*(x_k) i_{\lambda_k}^*(y_k)-i_{\lambda_k}^*(x_k) i_{\mu_k}^*(y_k))=x\cup y.
$$
Finally we claim that the dimension of $L'$ is $k-1$. Since $\dim((\ker (i_{\mu_k}^*\vert_L))=k-1$, we need to check that 
$res_1\oplus\cdots\oplus res_{k-1}$ restricted to $\ker i_{\mu_k}^*\vert_L$ is injective. Let 
$x\in \ker(res_1)\cap\cdots \cap \ker (res_{k-1})\cap \ker (i_{\mu_k}^*\vert_L)$, we want to prove that $x=0$.
Notice that $x\in H^1(\partial_kM;\mathfrak v)\cap L\cap \ker (i^*_{\mu_k})$.
Choose $y\in L$ such that $i_{\mu_k}^*(y)\neq 0$, this is possible because $i_{\mu_k}^*(L)\neq 0$. Then, using $x\in H^1(\partial_ kM;\mathfrak v)$,
Claim~\ref{lem:cupproduct} and Equation~(\ref{eqn:cupsum}), we obtain
$$
0=x\cup y= c_k (i_{\mu_k}^*(x) i_{\lambda_k}^*(y)-i_{\lambda_k}^*(x) i_{\mu_k}^*(y))= -c_k i_{\lambda_k}^*(x) i_{\mu_k}^*(y)
$$
for some $c_k\neq 0$. Since $i_{\mu_k}^*(y)\neq 0$, $i_{\lambda_k}^*(x)=0$. Therefore $x=0$.

Finally, the fact that injectivity fails if we consider only $k-1$ curves is clear once we know that the rank
of  $ H^1(\partial M;\mathfrak{v}) \to H^1(\mu_i;\mathfrak{v})$ is at most one. Indeed Lemma~\ref{lem:angles}
tells that this rank is at most one.
\end{proof}

\section{The function on the deformation space}
\label{sec:function}

Recall that $M$ denotes a compact manifold with boundary a  union  of $k>0$ tori and hyperbolic interior.
The goal of this section is to give a sufficient cohomological condition which guarantees that infinitely many fillings on $M$ are infinitesimally rigid. 

For this we shall define a real analytic function $f\co U\to\mathbf R$, where $U\subset \mathbf C^k$ is a parametrization of Thurston's slice. This function
is defined as the determinant of a matrix whose entries are pairings between homology and cohomology classes. This $f$ will depend of several choices, but 
its zero set is a well defined analytical subset of $U$. Another analytical subset of $U$ is defined by means of the dimension cohomology of $M$ with
twisted coefficients.  We shall prove that the Dehn fillings whose parameter $u$ is away from these analytical subsets are infinitesimally rigid. 
Moreover, when $M$ is infinitesimally projectively rigid with respect to the cusps, these subsets are proper.

For this we need several tools  for constructing a function on the deformation space. The first one is given by the following lemma.
All statements are up to taking a smaller neighborhood of $0$, $U\subset \mathbf C^k$.

\begin{lemma}
\label{lem:cohomologyclasses}
As in Section~\ref{sec:Thurstonsslice}, let  $U\subset \mathbf C^k$  be an open neighborhood of $0$ which parametrizes the deformations of the complete holonomy of the interior of $M$.

\begin{enumerate}
\item There exists a non-vanishing element 
$a_u^i\in \mathfrak{v}^{\rho_u(\pi_1(\partial_i M))}$ that varies analytically in $u\in U$.
\item There exists a family of cohomology classes 
$\{z_u^1,\ldots,z_u^k\}$ that define a basis for the image of 
$H^1(M; \mathfrak{v}_{\Ad\rho_{u}})\to H^1(\partial M; \mathfrak{v}_{\Ad\rho_{u}})$ 
and that varies analytically in $u\in U$.
\end{enumerate}
\end{lemma}

\begin{remark}
To vary analytically depends on the construction we take for cohomology, but we always think of an analytic map on a finite dimensional space of cocycles,
either in simplicial cohomology (fixing a triangulation and varying the bundle) or in group cohomology (fixing a generating set for the fundamental group). 
\end{remark}

\begin{proof}
The first assertion follows directly from Lemma~\ref{lem:invkilling} (ii).

For the second part we will use Lemma~\ref{lem:rankone} (ii).
The rank of $H^1(M; \mathfrak{v}_{\Ad\rho_{u}})\to H^1(\partial M; \mathfrak{v}_{\Ad\rho_{u}})$ is $k$. Hence it suffices to take a basis when $u=0$,
$\{ z_0^1,\ldots , z_0^k\}$ and then make it vary
in the kernel of $H^1(\partial M; \mathfrak{v}_{\Ad\rho_{u}})\to H^2(M,\partial M; \mathfrak{v}_{\Ad\rho_{u}})$, which is an analytic family of $k$--dimensional vector spaces.
\end{proof}

For $i=1,\ldots,k$ we consider the following $1$--cycle in the $i$th torus $\partial_iM$ of the boundary
$$
a_u^i\otimes \frac1{\vert p_i+q_i\tau_i\vert ^2}(p_i\mu_i+q_i\lambda_i)
$$
in simplicial homology.
This twisted cycle is the image of the untwisted cycle
$$\frac{p_i\mu_i+q_i\lambda_i}{ \vert p_i+q_i\tau_i\vert ^2}\in
H_1(\partial_i M, \mathbf R)$$
by the natural map 
$$
H_1(\partial_i M, \mathbf R) 
\xrightarrow{a^i_{u}\otimes\cdot} 
H_1(\partial_i M, \mathfrak{v}^{\rho_u(\pi_1(\partial_i M))})
\to H_1(\partial_i M, \mathfrak{v}_{\Ad\rho_u})$$ that consists in tensorizing by $a_u^i$ and composing with the map induced by the inclusion
of coefficients $\mathfrak{v}^{\rho_u(\pi_1(\partial_i M))}\to \mathfrak{v}$.

Let $\langle . \,, .\rangle$ denote the Kronecker pairing between homology and cohomology of $\partial M$ with coefficients in $\mathfrak{v}_{\Ad\rho_u}$
(see Subsection~\ref{subsection:Kronecker}).  We define
$$
f(u)= \det \bigg( \big(\langle z_u^i, a_u^j\otimes 
\frac{p_j\mu_j+q_j\lambda_j}{\vert p_j+q_j\tau_j\vert ^2}\rangle\big)_{ij}\bigg)
$$
where $p_i$ and $q_i$ are the generalized Dehn filling coefficients corresponding to $u\in U$ (see Section~\ref{sec:Thurstonsslice}).
If we view $z_u$ as a map on simplicial chains taking values on $\mathfrak{v}$, and $B$ denotes the Killing form, then
$$
f(u)= \det \bigg( B\big( z_u^i( 
\frac{p_j\mu_j+q_j\lambda_j}{\vert p_j+q_j\tau_j\vert ^2} ), a_u^j\big)\bigg).
$$
\begin{remark}
The function $f$ depends on several non-canonical choices. But we are only interested in the zero locus of $f$ and this set does not depend on the different choices involved in the definition of $f$.
Notice also that Lemma~\ref{lem:analytic} implies that $f$ is analytic and $f(0)=0$. {The proof of Proposition~\ref{prop:nonrigid} 
in Section~\ref{sec:orbi}  shows that the zero locus $f^{-1}(\{0\})$ of $f$ might be one dimensional and that in general $0\in f^{-1}(\{0\})$ is not an isolated point (see Section~\ref{sec:orbi}).}
\end{remark}

In the sequel let $u_{(\mathbf p, \mathbf q)}$ denote the parameter of the structure whose completion gives the Dehn filling with coefficients $(p_1,q_1),\ldots,(p_k,q_k)$ where $(p_i,q_i)$ are pairs of coprime  integers.

\begin{lemma}
\label{lemma:fnonzero}
If \begin{enumerate}[(i)]
\item $f(u_{(\mathbf p, \mathbf q)})\neq 0$ and
\item$\dim H^1(M, \mathfrak{v}_{\Ad\rho_{u_{(\mathbf p, \mathbf q)}}})=k$,
\end{enumerate}
 then $H^1(M_{(\mathbf p, \mathbf q)},\mathfrak{v}_{\Ad\rho_{u_{(\mathbf p, \mathbf q)}}})=0$.
\end{lemma}

\begin{proof}
In this proof the representation $\rho_{u_{(\mathbf p, \mathbf q)}}$ is fixed and we remove $\Ad\rho_{u}$ from notation.

Hypothesis (i) and (ii) imply that 
$$
\{a_u^1\otimes (p_1\mu_1+ q_1\lambda_1),\ldots ,a_u^k\otimes (p_k\mu_k+ q_k\lambda_k)\}
$$ is a basis for 
$H_1(M;\mathfrak{v})$. 
Hence for  $\gamma :=\gamma_1\cup\cdots\cup\gamma_k$, $\gamma_i=p_i\mu_i+q_i\lambda_i$,  
the following composition gives an 
isomorphism in homology:
\begin{equation*}
\bigoplus_{i=1}^k H_1(\gamma_i;\mathbf R)\to \bigoplus_{i=1}^kH_1(\gamma_i; \mathfrak{v}^{\rho_u(\pi_1(\partial_i M))})\to H_1(\gamma; \mathfrak{v})\to H_1(M;\mathfrak{v}).
\end{equation*}
Equivalently, we have an isomorphism in cohomology:
\begin{equation}
\label{ref:eqniso}
H^1(M;\mathfrak{v})
\to 
H^1(\gamma; \mathfrak{v})\to
\bigoplus_{i=1}^k
H^1(\gamma_i; \mathfrak{v}^{\rho_u(\pi_1(\partial_i M))})  
\to
\bigoplus_{i=1}^k
H^1(\gamma_i;\mathbf R). 
\end{equation}

Let $N$ denote a tubular neighborhood of the filling geodesics, so that $N=N_1\cup \cdots \cup N_k$ is the union of $k$ solid tori, $N\cup M$ is the closed manifold
$M_{(\mathbf p, \mathbf q)}$
and $N\cap M=\partial M$. We claim that the inclusions induce an isomorphism 
$$
H^i(M;\mathfrak{v})\oplus H^i(N; \mathfrak{v})\to H^i(\partial M; \mathfrak{v}) 
$$
for $i=0$ and $i=1$. 
Then by
Mayer-Vietoris,
$H^1(M_{(\mathbf p, \mathbf q)},\mathfrak{v})=0$ follows.

Let us check the claim. When $i=0$, 
$H^0(M;\mathfrak{v})\cong \mathfrak{v}^{Ad\rho_u(\pi_1M)}=0$, 
and the required isomorphism comes from the fact that 
$\pi_1(N_j)$ and $\pi_1(\partial_j M)$ have the same image under $\rho_u$ and hence the same invariant subspace.

When $i=1$, we notice that by Lemma~\ref{lem:rankone} 
$$H^1(\partial_i M,\mathfrak{v})=H^1(\partial_i M,\mathbf R)\otimes \mathfrak{v}^{\rho_u(\pi_1(\partial_i M))},$$ 
and $\dim \mathfrak{v}^{\rho_u(\pi_1(\partial_i M))}=1$,
by Lemma~\ref{lem:invkilling}. 
Similarly, 
$$
H^1(N_i,\mathfrak{v})=H^1(N_i,\mathbf R)\otimes \mathfrak{v}^{\rho_u(\pi_1(\partial_i M))}.
$$ 
Then the proof follows from isomorphism (\ref{ref:eqniso}) and the natural isomorphism induced by inclusions:
$$
H^1(\partial_ i M;\mathbf R)\cong H^1(N_i;\mathbf R)\oplus H^1(\gamma_ i;\mathbf R).
$$
\end{proof}

%

\begin{cor}
\label{cor:giverigidity}
If the generic dimension of $H^1(M;\mathfrak{v}_{\Ad\rho_u})$ is $k$ and if $f$ is non-constant in a neighborhood of $0$, then infinitely many Dehn fillings on $M$  are infinitesimally rigid. 
\end{cor}

\begin{proof}
The dimension of $H^1(M;\mathfrak{v}_{\Ad\rho_u})$ is bounded below by $k$ and upper semi-continuous on $u\in U$
(it is larger on a proper analytic subset). 
Hence the set of $u\in U$ where $\dim H^1(M;\mathfrak{v}_{\Ad\rho_u})\neq k$ or $f(u)=0$ is a proper analytic subset of $U$, 
and it misses infinitely many Dehn fillings by 
\cite[Lemme~4.4]{Porti}.
\end{proof}

\begin{cor}
\label{cor:cohmologydimU}
Assume that $M$ is infinitesimally projectively rigid relative to the cusps. Then 
for $u\in U$, the  dimension of $H^i(M;\mathfrak{v}_{\Ad\rho_u})$ is $k$ for $i=1,2$ and zero otherwise. 
\end{cor}

\begin{proof}
The dimension of $H^1(M;\mathfrak{v}_{\Ad\rho_u})$ is bounded below by $k$ and upper semi-continuous on $U$,
hence constant, because this dimension is reached at $u=0$. As $\rho_u$ is irreducible, $\mathfrak{v}_{\Ad\rho_u}$ has no 
non-trivial invariant element and therefore $H^0(M;\mathfrak{v}_{\Ad\rho_u})=0$. 
As the  Euler characteristic of $M$ vanishes,
and $M$ has the homotopy type of a two dimensional complex,  the 
dimension  of $H^i(M;\mathfrak{v}_{\Ad\rho_u})$ is also $k$ of $i=2$ and zero otherwise.
\end{proof}

For a collection of simple closed curves $\mu=\{\mu_1,\ldots \mu_k\}$, where $\mu_i\subset\partial_i M$ is non trivial in homology, 
the structure with cone angles $\alpha$ and meridians $\mu$ has parameter $u=(\alpha\mathbf i,\ldots,\alpha\mathbf i)\in U$, where $\mathbf i=\sqrt{-1}$. 
To simplify notation, we shall write $u=\alpha\mathbf i$, in particular the holonomy is written as
$\rho_{\alpha \mathbf i}$.

\begin{prop}
\label{prop:conerigid}
Assume that there exists a collection of simple closed curves as above 
$\mu\subset \partial M$ and some $\varepsilon>0$ so that, $\forall 0<\alpha<\varepsilon$,
$$
\dim H^1(M, \mu; \mathfrak{v}_{Ad \rho_{\alpha \mathbf i }})=3 k.
$$
Then infinitely many Dehn fillings are infinitesimally rigid. 
\end{prop}

\begin{proof} 
We first outline the proof.  
Our goal is to prove the proposition by applying Corollary~\ref{cor:giverigidity}.
The dimension of  $H^1(M; \mathfrak{v}_{Ad \rho_{u }})$ is $k$ when $u= \alpha \mathbf i$ by an argument on the long exact sequence in cohomology of the pair $(M,\partial M)$.
The same dimension count holds for a generic $u\in U$, by semicontinuity. In addition, the long exact sequence in cohomology also tells that 
$H^1(M; \mathfrak{v}_{Ad \rho_{u }})$ injects in $H^1(\mu; \mathfrak{v}_{Ad \rho_{u }})$, when $u= \alpha \mathbf i$. This
gives a basis for  $H^1(M; \mathfrak{v}_{Ad \rho_{u }})$ so that when we look at its image in  $H^1(\partial M; \mathfrak{v}_{Ad \rho_{u }})$ and we compute $f$,
we have that $f(u)\neq 0$, for $u=\alpha \mathbf i$.

Now we proceed with the details.
Since $\rho_{\alpha \mathbf i }(\mu_j)$ is a rotation of angle $0<\alpha<\pi$, by Lemma~\ref{lem:invhyperbolic}
$\dim H^0(\mu_j;\mathfrak{v}_{Ad\rho_{\alpha \mathbf i}})=\dim \mathfrak{v}^{Ad\rho_{\alpha \mathbf i}(\mu_j)}=3$, and therefore 
$$\dim H^0(\mu;\mathfrak{v}_{Ad\rho_{\alpha \mathbf i}})=3 k.$$

Then the long exact sequence of the pair $(M,\mu)$ starts as follows:
$$
0\to H^0(\mu, \mathfrak{v}_{Ad \rho_{\alpha \mathbf i}})\to  H^1(M,\mu, \mathfrak{v}_{Ad \rho_{\alpha  \mathbf  i}})\to H^1(M, \mathfrak{v}_{Ad \rho_{\alpha  \mathbf  i}})\to\cdots.
$$
Since $\dim  H^0(\mu; \mathfrak{v}_{Ad \rho_{\alpha  \mathbf i}})= \dim H^1(M,\mu, \mathfrak{v}_{Ad \rho_{\alpha  \mathbf i}})$, we have an inclusion
$$
0\to H^1( M; \mathfrak{v}_{Ad \rho_{\alpha  \mathbf i}})\to H^1(\mu; \mathfrak{v}_{Ad \rho_{\alpha  \mathbf i}}).
$$
The inclusion of $\mu$ in $M$ factors through $\partial M$, hence by Lemma~\ref{lem:rankone}, it follows that
$$
\dim H^1( M; \mathfrak{v}_{Ad \rho_{\alpha  \mathbf i}})=k,
$$
which is the first condition for applying Corollary~\ref{cor:giverigidity}, 
by upper semi-continuity of the dimension of $H^ 1$.

Moreover, using Lemma~\ref{lem:rankone} (i), it follows that
$$
H^1(M; \mathfrak{v}_{Ad\rho_{\alpha  \mathbf i}})\cong 
\bigoplus_{j=1}^k H^1(\mu_j;\mathbf R) \otimes \mathfrak{v}^{\rho_{\alpha  \mathbf i}(\pi_1(\partial_j M))}\,.$$
This implies that one can choose a basis $\{z_u^1,\ldots,z_u^k\}$ for $H^1(M; \mathfrak{v}_{Ad\rho_{\alpha  \mathbf i}})$, where
$z_u^j=\hat \mu_j\otimes a_{{\alpha\, \mathbf i}}^j$
and
$\hat \mu_j\in H^1(\mu_j;\mathbf Z)$ is the dual of the fundamental class in $H_1(\mu_j;\mathbf Z)$.
Thus, since $p_j=2\pi/\alpha$ and $q_j=0$, we get
$$
f(\alpha \mathbf i)=\frac{\alpha^k}{(2\pi)^k}B(a_{\alpha\, \mathbf i}^1, a_{\alpha\, \mathbf i}^1)\cdots B(a_{\alpha\, \mathbf i}^k, a_{\alpha\, \mathbf i}^k) \neq 0,
$$ as the Killing form on
$
\mathfrak{v}^{\rho_{{\alpha\, \mathbf i}}(\pi_1(\partial_j M))}
$ is non-degenerate.
\end{proof}

\begin{proof}[Proof of Theorem~\ref{thm:fillingstrong}]
As $M$ is infinitesimally projectively rigid, by Lemma~\ref{lem:almosteveryslope} 
we can choose a set of slopes $\mu=\mu_1\cup\cdots \cup \mu_k$, so that 
$$
0\to H^1(M;\mathfrak{v}_{\Ad\rho_0})\to  H^1(\mu;\mathfrak{v}_{\Ad\rho_0})
$$ is exact.
By the long exact sequence of the pair $(M,\mu)$, since $\dim \mathfrak{v}^{\Ad \rho_0(\mu_j)}=3$, this is equivalent to saying
that $\dim H^1(M,\mu;\mathfrak{v}_{\Ad\rho_0})=3 k$. By analyticity and upper semi-continuity of the dimension of the  cohomology,
the hypothesis of Proposition~\ref{prop:conerigid} holds true.
\end{proof}

\begin{proof}[Proof of Theorem~\ref{thm:fromonetomany}]
To simplify, we first assume that the infinitesimally projectively rigid Dehn filling can be connected to the complete structure by a family of cone manifold structures of cone angle $\alpha\in [0, 2\pi]$, 
where $\alpha=2\pi$
corresponds to the Dehn filling and $\alpha=0$ to the complete structure. Notice that this is the case of the Dehn fillings provided by Hodgson and Kerckhoff in their estimation of the size of the hyperbolic 
Dehn filling space \cite{HodgsonKerckhoff} (hence of all but at most 60 Dehn fillings).

Let $M_{(\mathbf p,\mathbf q)}$ be infinitesimally projectively rigid and let
$u_{(\mathbf p,\mathbf q)}\in U$ denote the parameter in the Thurston slice corresponding to the holonomy of the structure on $M$ induced by the Dehn filling. Let $V\subset U$ denote the domain of definition of the real analytic function $f\co V\to\mathbf R$.

As in the proof of  Lemma~\ref{lemma:fnonzero}, a Mayer-Vietoris argument gives that 
$$
\dim H^1(M; \mathfrak{v}_{\Ad\rho_{u_{(\mathbf p,\mathbf q)}}})=k.
$$ 
Moreover, if the parameter $u_{(\mathbf p,\mathbf q)}$ is contained in $V$ then $f({u_{(\mathbf p,\mathbf q)}})\neq 0$. 

A priori $V$ could be a smaller neighborhood of the origin and
$u_{(\mathbf p,\mathbf q)}\in U \smallsetminus V$ might happen.
The problem is that the cohomology classes 
$z_{u}^1,\ldots, z^k_{u}\in 
\image( H^1(M;\mathfrak{v}_{\Ad\rho_u})\to H^1(\partial M; \mathfrak{v}_{\Ad\rho_u}))$
could be linearly dependent or even not be defined outside $V$. 
To fix that,  we use the path of hyperbolic cone structures, that gives a segment in $U$, that we parametrize by the cone angle $\alpha\in [0,2\pi]$. Let $u_{\alpha}\in U$ denote the parameter
of the deformation space. By upper semi-continuity, 
$$
\dim H^1(M; \mathfrak{v}_{\Ad\rho_{u_{\alpha}}})=k=
\dim H^1(M; \mathfrak{v}_{\Ad\rho_{u_{2\pi}}})
$$
for almost all $\alpha\in [0,2\pi]$.
Let $\gamma_1,\ldots,\gamma_k$ denote the filling slopes.
By compactness, the segment $ [0,2\pi]$ is  covered by  intervals
$$
[0,2\pi]=[0,\alpha'_0)\cup (\alpha_1,\alpha_{1}')\cup\cdots \cup (\alpha_r,2 \pi],
$$
so that on each $(\alpha_i,\alpha_{i}')$, 
  there exist a family of  cohomology classes 
$$z_{\alpha}^1,\ldots, z_{\alpha}^k\in \image( H^1(M;\mathfrak{v}_{\Ad\rho_{u_{\alpha}}})\to H^1(\partial M; \mathfrak{v}_{\Ad\rho_{u_{\alpha}}}))$$  
that vary analytically on $\alpha\in(\alpha_i,\alpha_{i}')$ and are linearly independent
for each $\alpha\in(\alpha_i,\alpha_{i}')$, 
by Lemma~\ref{lem:cohomologyclasses}. On each interval we may use the cohomology
classes to construct functions 
$$
f_i:(\alpha_i,\alpha_{i}')\to \mathbf R
$$ similar to $f$, i.e. as the determinant of the matrix of Kronecker pairings between 
$z_{\alpha}^l$ and the homology class represented
by $a_{\alpha}^j\otimes \frac{\alpha}{2\pi} \gamma_j$.
The function $f_0:(\alpha_0,\alpha_{0}')\to \mathbf R$ can be chosen to be the restriction of $f:U\to\mathbf R$. In addition, the functions $f_i$ and $f_{i+1}$ may differ on 
$(\alpha_i',\alpha_{i+1})$, but $f_i(\alpha)=0$ if and only if $f_{i+1}(\alpha)=0$, for every $\alpha\in (\alpha_i',\alpha_{i+1})$. Since $f_r(2\pi)\neq 0$ and the $f_i$ are analytic,
$f_0$ and $f$ are non-zero. In addition, by upper semi-continuity, the generic dimension of $H^1(M; \mathfrak{v}_{\Ad\rho_{u}})$ is $k$, hence 
 we may apply 
Corollary~\ref{cor:giverigidity}. This finishes the proof when there is a segment of hyperbolic cone structures connecting $u_{(\mathbf p,\mathbf q)}$ to the origin.

In the general case, instead of considering a path of hyperbolic cone structures, which is a straight segment in $U$, it is sufficient to take a piece-wise analytic path  in $U$ connecting 
$u_{(\mathbf p,\mathbf q)}$ to the origin, and apply the previous argument. The only key point is that  we have to chose piecewise-analytic 
paths so that their non-smooth points are not in the vanishing locus of $f$,
which is always possible by genericity. 
\end{proof}

\section{Rigid slopes}
\label{sec:flexingslopes}

\begin{definition} Let $M^3$  be a cusped hyperbolic manifold of finite volume which is infinitesimally projectively rigid with respect to the cusps. Let $\gamma$ be a slope of $\partial_1 M$, We say that 
$\gamma$ is a \emph{rigid slope} if the map
$$
i_{\gamma}^*\co H^1(M;\mathfrak{v}_{\Ad\rho_0})\to H^1(\gamma; \mathfrak{v}_{\Ad\rho_0})
$$
is nontrivial.
\end{definition}

\begin{prop}
\label{prop:flexingslope}
Let $M^3$  be a cusped hyperbolic manifold of finite volume which is infinitesimally projectively rigid with respect to the cusps 
and let $\mu,\lambda\in\partial_1 M$ be a pair of simple closed curves generating the fundamental group of 
$\partial_1 M$.
Let $(p_n ,q_n)\in\mathbf Z^2$ be a sequence of coprime integers lying on a line $a \, p_n+ b\, q_n = c$. If $\gamma= -b \mu+a \lambda$ is a rigid slope, then $M^3_{(p_n,q_n),\infty,\cdots,\infty}$
is infinitesimally rigid with respect to the cusps for $n$ large enough.
\end{prop}

\begin{proof}
After changing the basis in homology, the curves $\mu$ and 
$\lambda$ are chosen such that  $a=1$, $b=0$, i.e.\ 
$\lambda=(0,1)$ is the rigid slope. 
We also may assume $(p_n,q_n)=(c,n)$. 

Let us consider the path
$$ s\mapsto \begin{cases} (c,\frac 1s) & \text{ if } s\neq 0\\
\infty & \text{ if } s =0 \end{cases}$$ 
in the parameter space. Denote by $u(s)$ the corresponding point in the deformation space. 

\begin{lemma}
The path $u(s)$ is a real analytic on  ${s\in (-\varepsilon, \varepsilon)}$.
\end{lemma}

\begin{proof}
Setting $\tau(u)=v(u)/u$, from $p\, u+q\, v= u( c+\frac{1}{s}\tau(u))=2\pi\mathbf i$ 
we can write
$$
u({s\, c+ \tau(u)})
={s\, 2\pi\mathbf i}.
$$ 
Since $\tau(0)\neq 0$ and $\tau$ is analytic on $u$, this allows to define $u$ as analytic function 
on $s$, by applying the analytic version of the implicit function theorem.
\end{proof}

Let $\theta_u\in \image\big(H^1(M;\mathfrak{v}_{Ad\rho_u})\to H^1(\partial_1M;\mathfrak{v}_{Ad\rho_u})\big)$ be an analytic family of cohomology classes, so that  $i_{\lambda}^*(\theta_0)\neq 0$. This is always possible since $i_{\lambda}^*$ factors through $H^1(\partial_1M;\mathfrak{v}_{Ad\rho_u})$.

The two cohomology classes
$z_{\mu}, z_{\lambda}\in H^1(\partial_1 M;\mathfrak{v}_{Ad\rho_0})$ as defined in the proof of Lemma~\ref{lem:angles} satisfy  
$i_{\mu}^*(z_{\lambda})=i_{\lambda}^*(z_{\mu})=0$,
$i_{\mu}^*(z_{\mu})\neq 0$, and $i_{\lambda}^*(z_{\lambda})\neq 0$.
Hence we may assume that 
$$
\theta_0= z_{\lambda}+ \beta z_{\mu},\qquad\textrm{ for some }
\beta\in\mathbf R.
$$

Let also $a_{u(s)}\in \mathfrak{v}^{\Ad\rho_{u(s)}(\pi_1(\partial_1 M))}$ be an analytic family of invariant elements, with $a_0\neq 0$.
As in Lemma~\ref{lem:cohomologyclasses}, we want to see that for $s>0$, the following function does not vanish:
\begin{align*}
f(s) := &\bigg\langle \theta_{u(s)}, a_{u(s)}\otimes \frac{ c\mu+\frac1s\lambda}{| c +\frac1s\tau|^2}
\bigg\rangle\\ = & \frac{s}{| s\,c+\tau|^2}
\big\langle 
\theta_{u(s)} ( s\,c\mu+\lambda ), a_{u(s)}
\big\rangle.
\end{align*}

Notice that it follows from the proof of Lemma~\ref{lem:invkilling} that for small $s$, $s\neq0$, the restriction of the Killing  form on the subspace $\mathfrak{v}^{\Ad\rho_{u(s)}(\pi_1(\partial_1 M))}$ is positive definite i.e.\ $B(a_{u(s)},a_{u(s)})>0$ for sufficiently small $s\neq 0$.
\begin{lemma}
\label{lem:Bsto0}
If $\Vert a_{u(s)} \Vert=B(a_{u(s)},a_{u(s)})^{1/2}$, then
$$
\lim\limits_{s\to 0}\frac{B( \theta_{u(s)}(\lambda), a_{u(s)})}
{\Vert a_{u(s)}\Vert}=16
\quad\textrm{ and }\quad
\lim\limits_{s\to 0}\frac{B( \theta_{u(s)}(\mu), a_{u(s)})}{\Vert a_{u(s)}\Vert}=16\beta.
$$
\end{lemma}

Assuming the lemma we obtain
\[
\frac{f(s)}{s\, \Vert a_{u(s)}\Vert}=
\frac1{| s\, c+\tau | ^2}
\left(
\frac{B( \theta_{u(s)}(\lambda), a_{u(s)})}{\Vert a_{u(s)}\Vert}
+s\, c
\frac{B( \theta_{u(s)}(\mu), a_{u(s)})}{\Vert a_{u(s)}\Vert}
\right)\]
and hence
\[
\lim_{s\to0} \frac{f(s)}{s\, \Vert a_{u(s)}\Vert} =
\frac{16}{|\tau_0|^2}\,.
\]
Hence $f(s)\neq 0$ for $s\neq 0$. Moreover, since the dimension of $H^1(M; \mathfrak{v}_{Ad\rho_{u}})$
is upper semi-continuous with respect to $u$, it still satisfies $\dim (H^1(M; \mathfrak{v}_{Ad\rho_{u(s)}}))= k$. By analyticity
those conditions are satisfied for all but finitely many $s$, hence we may apply  Lemma~\ref{lemma:fnonzero}.

This concludes the proof 
of Proposition~\ref{prop:flexingslope} assuming Lemma~\ref{lem:Bsto0}.
\end{proof}

Before proving Lemma~\ref{lem:Bsto0}, we still need a further computation. 
Let $w_0\in \mathfrak{su}(3,1)$ denote 
$$
w_0=\frac{\mathbf i}{2} V_0,\qquad \text{ where }
V_0=\begin{pmatrix}
        1 & & & \\
	& 1 & & \\
	& & -1 & \\
	& & & -1
   \end{pmatrix}.
$$
Note that $w_0$ is contained in 
$\mathfrak g_0\subset\mathfrak{su}(3,1)$ which is the Lie algebra of the stabilizer of $[v_\pm]\in\partial_{\infty} \mathbf H^3_{\mathbf C}$.

\begin{lemma} 
\label{lem:vutaylor} The invariant element $a_u\in \mathfrak{v}^{\rho_u(\pi_1(\partial_1 M))}$ can be chosen such that:
$$a_u= p(u)+ 4\left|\sinh^2\frac u2\right| V_0$$
where $p(u)$ is an infinitesimal parabolic transformation.
\end{lemma}

\begin{proof}
Since $w_0$ is invariant by the stabilizer $G_0$ 
for $u\neq 0$, $a_u$ can be obtained by conjugating $w_0$, and then by normalizing the result so that the limit exists if $u$ tends to $0$.

Recall that in the Heisenberg model the subgroup of real parabolic representations corresponds to $\mathbf R^2\times\{0\}\subset \mathcal H_- \subset G_- = \mathcal H_- \rtimes(U(2)\times\mathbf R)$. Note also that $w_0$ is the image of $\mathbf i I_2$ under the canonical inclusion $\mathfrak u(2)\hookrightarrow \mathfrak{su}(3,1)$.

Suppose that  $(x,y,0)\in\mathbf R^2\times\{0\}$ is the second  fixed point  of $\rho_u(\pi_1 \partial M)$.  In the notation of 
$PSL_2(\mathbf C)$ we have
$$
\rho_u(\mu)=\pm\begin{pmatrix} e^{u/2} & 1 \\ 0 & e^{-u/2} \end{pmatrix},
$$
hence
$$
x+\mathbf i y = \frac{-1}{2\sinh(u/2)}.
$$
Using the formalism of $G_-$, the conjugate of $w_0$ we are looking for is:
\begin{align*}
\Ad_{(x,y,0)}
\begin{pmatrix} \mathbf i & 0 \\ 0 &\mathbf i \end{pmatrix}
&= \frac{d}{dt} (x,y,0)
\begin{pmatrix} e^{\mathbf i t} & 0 \\ 0 &e^{\mathbf i t} \end{pmatrix} (-x,-y,0)\bigg|_{t=0}\\
&= \frac{d}{dt} (x,y,0) (-xe^{\mathbf i t},-ye^{\mathbf i t},0)\begin{pmatrix} e^{\mathbf i t} & 0 \\ 0 &e^{\mathbf i t} \end{pmatrix}\bigg|_{t=0}\\
&= \frac{d}{dt} 
\big(x(1-e^{\mathbf i t}) ,y(1-e^{\mathbf i t}),(x^2+y^2)\sin(t)\big) \begin{pmatrix} e^{\mathbf i t} & 0 \\ 0 &e^{\mathbf i t}\end{pmatrix} \bigg|_{t=0}\\
&= \big(- \mathbf i x , - \mathbf i y ,(x^2+y^2) \big) + 
\begin{pmatrix} \mathbf i & 0 \\ 0 & \mathbf i\end{pmatrix}\,.
\end{align*}
Under the inclusion $\mathfrak g_-\hookrightarrow\mathfrak{su}(3,1)$ this element is written as
\begin{equation*}
\begin{pmatrix}
        \frac{\mathbf i}2 & & & \\
	& \frac{\mathbf i}2 & & \\
	& & -\frac{\mathbf i}2 & \\
	& & & -\frac{\mathbf i}2
   \end{pmatrix}
-
\mathbf i \begin{pmatrix}
                             0 & 0 & x & y \\
			      0 & 0 & x & y \\
			      -x & x & 0 & 0 \\
			      -y & y & 0 & 0
                            \end{pmatrix}
+
\mathbf i (x^2+y^2) \begin{pmatrix}
                             0 & 0 & 0 & 0 \\
			      0 & 0 & 0 & 0 \\
			      0 & 0 & 1 & -1 \\
			      0 & 0 & 1 & -1
                            \end{pmatrix}.
\end{equation*}
Hence $  \Ad_{(x,y,0)}(w_0) = w_0 +\textrm{Parabolic}$.

Now $x^2+y^2=1/\vert 4\sinh^2(u/2)\vert$ and in order to obtain an invariant matrix which converges when $u\to 0$ we take
$$
a_u=  -\mathbf i 4 \left|\sinh^2\frac u2\right| \Ad_{(x,y,0)}(w_0)=   4 \left|\sinh^2\frac u2\right|  V_0 +\textrm{Parabolic}
$$
and the lemma is clear.
\end{proof}

\begin{proof}[Proof of Lemma~\ref{lem:Bsto0}]
Using Lemmas~\ref{lem:vutaylor} and \ref{lem:orthogonalparabolic} we obtain:
\begin{eqnarray*}
B(a_u,a_u)^{1/2} & = &  4 \left|\sinh^2 (u/2) \right| \, B(V_0,V_0)^{1/2} =   8 \left|\sinh^2  (u/2)\right|; \\
B( \theta_{u(s)}(\lambda),a_{u(s)}) & = & B(  \theta_{u(s)}(\lambda), V_0) \,4  \left |\sinh^2(u/2) \right| .
\end{eqnarray*}
Hence 
$$
\frac {B( \theta_{u(s)}(\lambda),a_{u(s)})} {\Vert a_{u(s)}\Vert }= 
\frac 1 2 B( \theta_{u(s)}(\lambda), V_0)\to \frac12
B( \theta_{u(0)}(\lambda), V_0) \textrm{ as }
s\to 0, $$
and 
\[
B( \theta_{u(0)}(\lambda), V_0)= B(z_{\lambda}(\lambda),V_0)=\\
B(a_\mu,W_0)
= 32\,.
\]
A similar computation holds for $\theta_{u(s)}(\mu)$.
\end{proof}

\section{Examples}
\label{sec:examples}

In this section we compute two examples, the figure eight knot and the Whitehead link exteriors. We start introducing some notation. 
Let $\mathbf{x}\in\mathbf{R}^4$ be a column vector. As in Section~\ref{subsectioncomplex} we will use the following notation: $\mathbf{x}^*=\mathbf{x}^t J$. Then for all 
$\mathbf x, \mathbf y\in\mathbf R^4$ we have that
$\mathbf x \mathbf y^* + \mathbf y\mathbf x^*\in\mathfrak v$. 
In the sequel we will make use of the following basis 
$\{ v_1,\ldots, v_9\}$ of $\mathfrak v$:
\[ 
v_i = \mathbf e_i \mathbf e_i^* + \mathbf e_4 \mathbf e_4^* \qquad \text{ for $i=1,\ldots,3$},
\]
and
\begin{eqnarray*}
v_4  = \mathbf e_1 \mathbf e_2^* + \mathbf e_2 \mathbf e_1^*, &
v_5  = \mathbf e_1 \mathbf e_3^* + \mathbf e_3 \mathbf e_1^*, &
v_6  = \mathbf e_1 \mathbf e_4^* + \mathbf e_4 \mathbf e_1^*,\\
v_7  = \mathbf e_2 \mathbf e_3^* + \mathbf e_3 \mathbf e_2^*, &
v_8  = \mathbf e_2 \mathbf e_4^* + \mathbf e_4 \mathbf e_2^*, &
v_9  = \mathbf e_3 \mathbf e_4^* + \mathbf e_4 \mathbf e_3^*\,.
\end{eqnarray*}

\subsection{The figure eight knot}

In this section we explain the computations to show that the figure eight knot exterior is infinitesimally projectively rigid.

Let $\Gamma$ be the fundamental group of the figure eight knot exterior. We fix a presentation of $\Gamma$:
\begin{equation}\label{eq:presfig8}
\Gamma=\langle x,\, y \mid x y^{-1} x^{-1} y x y^{-1} x y x^{-1} y^{-1} \rangle. 
\end{equation}
where $x$ and $y$ represent meridians. 

By Corollary~\ref{cor:dim}, it suffices to show that $\dim H^1(\Gamma,\mathfrak{v}_{\Ad\rho_0})=1$.

We start with a holonomy representation of the complete structure in $SL_2(\mathbf C)$
\cite{Riley}:
$$
x\mapsto   \left( \begin {array}{cc} 1&1\\\noalign{\medskip}0&1\end {array}
\right)
\qquad
y\mapsto  \left(       \begin {array}{cc} 1&0\\\noalign{\medskip}\frac{1-i\sqrt {3}}2 &
1\end {array}         \right),
$$

Using for instance the construction described in \cite{CLTEM}, the representation in $PSO(3,1)$ is given by:
$$
\rho_0(x)= 
\begin{pmatrix}
1&0&0&0\\
0&1&-1&1\\
0&1&1/2&1/2\\
0&1&-1/2&3/2
\end{pmatrix}
%
\qquad
\rho_0(y)=  
\begin{pmatrix}
1&0& \sqrt {3}/2& \sqrt {3}/2\\
0&1&1/2&1/2\\
-\sqrt {3}/2&-1/2&1/2&-1/2\\
\sqrt {3}/2&1/2&1/2&3/2
\end{pmatrix}
$$
Notice that the holonomy of $x$ and $y$ have a fixed point in the light cone, which are respectively:
$$
v_+ = 
\begin{pmatrix}
0 \\ 0 \\ 1 \\ 1
\end{pmatrix}
\quad
\textrm{ and }
\quad
v_- =
\begin{pmatrix}
0 \\ 0 \\ -1 \\ 1
\end{pmatrix}.
$$
With respect to the basis $\{v_1,\ldots,v_9\}$ for $\mathfrak{v}$
the adjoint representation is given by:
$$
\Ad\rho_0(x)= \left(
 \begin {matrix}
1&0&0&0&0&0&0&0&0\\[1ex]
1&2&2&0&0&0&-2&2&-2\\[1ex]
\frac14&\frac54&\frac12&0&0&0&1&1&\frac12\\[1ex]
0&0&0&1&-1&-1&0&0&0\\[1ex]
0&0&0&1&\frac12&-\frac12&0&0&0\\[1ex]
0&0&0&-1&\frac12&\frac32&0&0&0\\[1ex]
\frac12&\frac32&0&0&0&0&-\frac12&\frac32&0\\[1ex]
\frac32&\frac52&2&0&0&0&-\frac32&\frac52&-2\\[1ex]
\frac34&\frac74&\frac12&0&0&0&0&2&\frac12
\end{matrix} \right)  
$$
and
$$\small
\Ad\rho_0(y)= \left(
\begin{matrix}
\frac74&\frac34&\frac32&0&\sqrt {3}&-\sqrt {3}&0&0&\frac32\\[1ex]
\frac14&\frac54&\frac12&0&0&0&1&1&\frac12\\[1ex]
1&\frac12&\frac12&\frac{\sqrt {3}}{2}&-\frac{\sqrt {3}}2&
-\frac{\sqrt{3}}2&-\frac12&\frac12&-\frac12\\[1ex]
\frac{\sqrt {3}}4&\frac{\sqrt {3}}4&\frac{\sqrt {3}}2&1&1/2&-1/2&
\frac{\sqrt {3}}2&\frac{\sqrt {3}}2&\frac{\sqrt {3}}2\\[1ex]
-\frac{3\sqrt {3}}4&-\frac{\sqrt {3}}4&0&-\frac12&-\frac14&
\frac54&-\frac{\sqrt {3}}4&-\frac{\sqrt {3}}4&0\\[1ex]
-\frac{5\sqrt {3}}4&-\frac{3\sqrt {3}}4&-\sqrt {3}&-\frac12&-\frac54&\frac94&-\frac{\sqrt {3}}4&-\frac{\sqrt {3}}4&-\sqrt {3}\\[1ex]
-\frac14&-\frac34&0&-\frac{\sqrt {3}}2&-\frac{\sqrt {3}}4&
\frac{\sqrt {3}}4&\frac14&-\frac34&0\\[1ex]
\frac34&\frac54&1&\frac{\sqrt {3}}2&\frac{\sqrt {3}}4&
-\frac{\sqrt {3}}4&\frac34&\frac74&1\\[1ex]
-\frac32&-1&-\frac12&-\frac{\sqrt {3}}2&0&\sqrt {3}&0&-1&\frac12
\end {matrix}\right).
$$
The cohomology group $H^1(\Gamma;\mathfrak{v})$ is computed as the quotient $Z^1/B^1$, where $Z^1=Z^1(\Gamma\; \mathfrak{v}_{\Ad\rho_0})$ is the space of cocycles and 
$B^1=B^1(\Gamma\; \mathfrak{v}_{\Ad\rho_0})$ the space of coboundaries, cf. Subsection~\ref{subsection:groupcohomology}.

Since $\mathfrak{v}$ has no element globally invariant by $\Gamma$, $\dim B^1=\dim \mathfrak{v}= 9$. We claim that $\dim Z^1=10$. To compute this dimension, we use the 
isomorphism of $\mathbf R$--vector spaces:
$$
\begin{array}{rcl}
Z^1 & \leftrightarrow &\{ (a,b)\in \mathfrak{v}^2\mid \tfrac {\partial w}{\partial x}\cdot a + \tfrac {\partial w}{\partial y}\cdot b =0 \} \\
d & \leftrightarrow & (d(x),d(y))
\end{array},
$$
where $w=x y^{-1} x^{-1} y x y^{-1} x y x^{-1} y^{-1}$ is the relation in the presentation of $\Gamma$, and
$\frac {\partial w}{\partial x}$, $\frac {\partial w}{\partial y}$ stand for the Fox derivatives \cite{LubotzkyMagid}:
\begin{eqnarray*}
\frac {\partial w}{\partial x} & = &
  1-xy^{-1}x^{-1}+x y^{-1} x^{-1} y+ y x y^{-1} x^{-1}- y,
\\
\frac {\partial w}{\partial y} & = & 
   -x y^{-1}+ x y^{-1} x^{-1}- y x y^{-1} x^{-1}+y x y^{-1} -1.
\end{eqnarray*}
Thus, $Z^1$ is isomorphic to the kernel of the linear map from $\mathfrak{v}\times \mathfrak{v}$ to $\mathfrak{v}$ with matrix:
\begin{equation}
\label{eqn:boundary}
\left(
\Ad\rho_0(\frac {\partial w}{\partial x}) \; ,\;
\Ad\rho_0(\frac {\partial w}{\partial y})
\right).
\end{equation}

One can check that this matrix has rank 8, by means of an elementary but tedious computation.
Hence $\dim Z^1=10$, as claimed.

To prove Proposition~\ref{prop:fig8explicit} we need to show:

\begin{remark}\label{rem:flex}
The longitude is  a rigid slope.
\end{remark}

With this remark, Proposition~\ref{prop:fig8explicit} is just an application of Proposition~\ref{prop:flexingslope}.
To prove that the longitude is a rigid slope, we need to analyze more carefully the previous computation.

By looking at the kernel of matrix~(\ref{eqn:boundary}), we choose one  cocycle $d$ determined by:
$$
d(x)= \begin {pmatrix}
0&0&0&0\\
0&0&-3&-1\\
0&-3&0&0\\
0&1&0&0
\end {pmatrix}
\quad\textrm{ and }\quad
d(y)= 0
\,.$$
Let $l= y x^{-1} y^{-1} x^2 y^{-1} x^{-1} y$ be the longitude that commutes with $x$. Then, by Fox calculus,
$$
d(l)=\begin {pmatrix} 
60&-4\,\sqrt {3}&60\,\sqrt {3}&-68\,\sqrt {3}\\
-4\,\sqrt {3}&-4&-12&12\\
60\,\sqrt {3}&-12&178&-206\\
68\,\sqrt {3}&-12&206&-234
\end {pmatrix}\,.
$$
To see that $d$ restricted to the cyclic group generated by $l$ is not a coboundary, following the proof of Lemma~\ref{lem:angles},
we must find an invariant element 
$a\in \mathfrak{v}^{\Ad\rho_0(l)}$ such that 
$B(d(l),a)\neq 0$. Since:
$$
\rho_0(l)= \begin {pmatrix} 
1&0&-2\,\sqrt {3}&2\,\sqrt {3}\\
0&1&0&0\\
2\,\sqrt {3}&0&-5&6\\
2\,\sqrt {3}&0&-6&7
\end {pmatrix} ,
$$
following again the proof of Lemma~\ref{lem:angles}, we choose
$$
a=\begin{pmatrix}
  -1 & & & \\
   & 3 & & \\
   & & -1 & \\
   & & & -1
 \end{pmatrix},
$$
and we have that $B(d(l),a)=-16\neq 0$.

\subsection{Orbifolds with the figure eight knot as the singular locus}\label{sec:orbi}

Let ${\mathcal O}_n$ denote the orbifold with underlying space $S^3$, singular locus $\mathrm{Sing}({\mathcal O}_n)$ the 
figure eight knot and ramification index $n$. The orbifold ${\mathcal O}_n$ is hyperbolic for $n\geq4$. 
Note that the orbifold ${\mathcal O}_n$ has a finite cyclic covering $\widetilde{{\mathcal O}_n} \to {\mathcal O}_n$ where 
$M_n :=\widetilde{{\mathcal O}_n}$ is the so called Fibonacci manifold which is widely studied in the literature \cite{HKM}.

The aim of this 
subsection is to prove Proposition~\ref{prop:nonrigid}, which states that ${\mathcal O}_n$ is not locally 
projectively rigid for sufficiently large $n$, and that its deformation space is a curve.
This will be proved in Paragraph~\ref{subsubsection:deform}, after three paragraphs of preliminary results.

As before,  
$\Gamma_0 := \Gamma = \pi_1({\mathcal O}_n\setminus \mathrm{Sing}({\mathcal O}_n))$ 
denotes the fundamental group of the figure eight knot exterior, so that
\[ 
\Gamma_{1/n} := \pi_1^\mathrm{orb}({\mathcal O}_n) \cong \Gamma / \langle m^n\rangle,
\]
for $m\in\Gamma$ representing a meridian. Note that there exists an exact sequence
\[ 0\to\pi_1(M_n)\to \pi_1^\mathrm{orb}({\mathcal O}_n) \to \mathbf Z / n\mathbf Z \to 0\,.\]

\subsubsection{A finite order automorphism of $\Gamma_0$}

The figure eight knot is amphicheiral and hence there exists an automorphism of $\Gamma_0$ preserving  the longitude and sending the meridian to its inverse. 
Such an automorphism $\varphi_0\co\Gamma_0\to\Gamma_0$ is given by
\[ \varphi_0(x) = x^{-1} \text{ and } \varphi_0(y)= y x^{-1} y^{-1} x y^{-1}.\]
By direct calculation using Presentation~(\ref{eq:presfig8}) and 
the meridian/longitude pair
$m=x$ and $l= y x^{-1} y^{-1} x^2 y^{-1} x^{-1} y$, one checks that $\varphi_0$ is an automorphism and that
 \[ \varphi_0(m)= m^{-1} \text{ and }\varphi_0(l) =l.\]
Hence $\varphi_0$ induces automorphisms
\[\varphi_{1/n}\co \Gamma_{1/n}\to\Gamma_{1/n}.\]

Let $\rho_0\co\Gamma_0\to PSO(3,1)$ and
$\rho_{1/n}\co\Gamma_{1/n}\to PSO(3,1)$ denote the holonomy representations. Then by Mostow--Prasad rigidity there exists a 
unique element $A_{1/n}\in PSO(3,1)$ such that
\begin{equation}\label{eq:Ad}
\rho_{1/n}\circ\varphi_{1/n} = \Ad_{A_{1/n}}\circ \rho_{1/n} 
\end{equation}
for $n\geq 4$, including $0 = 1/\infty$.

For any group homomorphism $\varphi\co\Gamma\to \Gamma'$ and any $\Gamma'$--module $\mathfrak a'$ we denote by 
$  \mbox{}^{\varphi}\mathfrak a'$ the 
$\Gamma$--module with underlying set $\mathfrak a'$ and the $\Gamma$ action $\gamma\circ a' = \varphi(\gamma)\circ a'$. 
It is easy to check that $\varphi$ induces a map
\[ f^*\co H^*(\Gamma',\mathfrak a')\to H^*(\Gamma,  \mbox{}^{\varphi}\mathfrak a')\]
(see \cite[III.8]{Bro82}).
Now any $\Gamma$--module $\mathfrak a$ and any morphism of 
$\Gamma$--modules 
$\alpha\co  \mbox{}^{\varphi}\mathfrak a'\to \mathfrak a$ there is an induced map in cohomology  
$(\varphi,\alpha)^*\co H^*(\Gamma,\mathfrak a)\to H^*(\Gamma,\mathfrak a)$ given by
\[ (\varphi,\alpha)^* = \alpha_* \circ \varphi^* .\]
By Proposition III.8.3 from \cite{Bro82}, inner automorphisms of $\Gamma$ induce the identity
on cohomology.

Now Equation~(\ref{eq:Ad}) tells us that 
$\Ad_{A_{1/n}^{-1}}\co \mbox{}^{\varphi_{1/n}} \mathfrak{v}_{\rho_{1/n}}
\to \mathfrak{v}_{\rho_{1/n}}$ is a $\Gamma_{1/n}$--module morphism  and hence there is an induced map
\[ \varphi_{1/n}^* := (\varphi_{1/n},\Ad_{A_{1/n}^{-1}})^*\co 
H^1(\Gamma_{1/n},\mathfrak{v}_{\rho_{1/n}}) \to
H^1(\Gamma_{1/n},\mathfrak{v}_{\rho_{1/n}}).\]
To work explicitly with this map, we work with cocycles $Z^1(\Gamma_{1/n},\mathfrak{v}_{\rho_{1/n}})$ 
i.e. maps $z:\Gamma_{1/n}\to \mathfrak{v}_{\rho_{1/n}}$ satisfying the cocycle relation
(Subsection~\ref{subsection:groupcohomology}). We also denote $ \varphi_{1/n}^* :Z^1(\Gamma_{1/n},\mathfrak{v}_{\rho_{1/n}})\to Z^1(\Gamma_{1/n},\mathfrak{v}_{\rho_{1/n}})$
the induced map on cocycles, that is given by
 $$
\varphi_{1/n}^*(z) = \Ad_{A_{1/n}^{-1}}\circ z\circ \varphi_{1/n}, \qquad \forall z\in Z^1(\Gamma_{1/n},\mathfrak{v}_{\rho_{1/n}}).
$$

In the sequel we shall compute the action of $\varphi^*_{1/n}$ first on the homology $H^*(\partial M,\mathfrak{v}_{\rho_{1/n}}) $ and then we shall deduce its action on $H^*(\Gamma_{1/n},\mathfrak{v}_{\rho_{1/n}})$. 

For $4\leq n <\infty$, we have a natural isomorphism 
\[ H^*(\partial M,\mathfrak{v}_{\rho_{1/n}}) \cong
H^*(\partial M,\mathbf{R}) \otimes 
\mathfrak{v}^{\rho_{1/n}(\pi_1 \partial M)} \]
(see Lemma~\ref{lem:rankone}). For $n=\infty$ Lemma~\ref{lem:angles} applies and hence
\[
i_l^*\oplus i_m^*\co H^1(\partial M,\mathfrak{v}_{\rho_{0}})\to
H^1(l,\mathfrak{v}_{\rho_{0}})\oplus H^*(m,\mathfrak{v}_{\rho_{0}})
\]
is injective. Moreover $\rank (i_l^*)=\rank (i_m^*)=1$.

In the sequel let $\varphi^*\co H^*(\partial M,\mathbf{R})\to H^*(\partial M,\mathbf{R})$ denote the the map induced in the untwisted cohomology with real coefficients.

\begin{lemma}\label{lem:2}
For $n<\infty$, with respect to the isomorphism
$H^*(\partial M,\mathfrak{v}_{\rho_{1/n}}) \cong
H^*(\partial M,\mathbf{R}) \otimes 
\mathfrak{v}^{\rho_{1/n}(\pi_1 \partial M)}
$,
 the isomorphism $\varphi^*_{1/n}$ on  cohomology
 is given by
\[ \varphi^*_{1/n}= \varphi^* \otimes 
\Id_{\mathfrak{v}^{\rho_{1/n}(\pi_1 \partial M)}}.
\]
For $n=\infty$, we have
\[ i_l^*\circ \varphi^*_0 = i^*_l \text{ and }
i_m^*\circ \varphi^*_0 = - i^*_m\,.\]
\end{lemma} 
\begin{proof}
If $n\geq 4$ then $\rho_{1/n}(m)$ is an elliptic element and 
$\rho_{1/n}(l)$ is a pure hyperbolic translation. This can be seen for example by using the trace identity
\[ \tr\rho(l) = \tr^4 \rho(m) - 5 \tr^2 \rho(m) + 2,\]
which holds for every irreducible representation $\rho\co\Gamma\to SL(2,\mathbf C)$ (see for example \cite[p.~113]{Porti}).
Hence up to conjugation we may assume that
\begin{align*}
\rho_{1/n} (m) &=
\begin{pmatrix}
\cos (2\pi/n) & -\sin (2\pi/n) & 0 & 0 \\
\sin (2\pi/n) & \cos( 2\pi/n) & 0 & 0 \\
0 & 0 & 1 &0 \\
0 & 0 & 0 & 1
\end{pmatrix} \intertext{ and }
\rho_{1/n} (l) &=
\begin{pmatrix}
1 & 0 & 0 &0 \\
0 & 1 & 0 & 0\\
0 & 0 & \cosh (\lambda_n) & \sinh (\lambda_n)  \\
0 & 0 & \sinh (\lambda_n) & \cosh(\lambda_n)\\
\end{pmatrix},
\end{align*}
where $\lambda_n$ is the translation length of the holonomy of $l$, which is the  length of the geodesic singular locus. 
With this normalization we obtain 
\[ \mathfrak{v}^{\rho_{1/n}(\pi_1 \partial M)} = \Big\langle
\left(\begin{smallmatrix}
1 & 0 & 0 &0 \\
0 & 1 & 0 & 0\\
0 &0 & -1 & 0 \\
0 &0 & 0 & -1 
\end{smallmatrix}\right)\Big\rangle
\]
and
\[
A_{1/n}= 
\left(\begin{smallmatrix}
1 & 0 & 0 &0 \\
0 & -1 & 0 & 0\\
0 &0 & 1 & 0 \\
0 &0 & 0 & 1 
\end{smallmatrix}\right)
\begin{pmatrix}
R_\alpha &0\\
0 & T_\eta 
\end{pmatrix}
\]
where $R_\alpha$ is a rotation of angle $\alpha\in\mathbf R$ and
$T_\eta$ is a hyperbolic translation of length $\eta\in\mathbf R$.
The actual values of $\alpha$ and $\eta$ are not needed since the above form of $A_{1/n}$ already implies that it acts trivially on
$\mathfrak{v}^{\rho_{1/n}(\pi_1 \partial M)}$ i.e.\
\[
\Ad_{A_{1/n}}\big|_{\mathfrak{v}^{\rho_{1/n}(\pi_1 \partial M)}} 
=\Id_{\mathfrak{v}^{\rho_{1/n}(\pi_1 \partial M)} },
\]
and the first assertion of the lemma follows.

In order to prove the second assertion recall that 
\[\rho_0(m)=\rho_0(x) =
\exp\left(\begin{smallmatrix}
0 & 0 & 0 &0 \\
0 & 0 & -1 & 1\\
0 & 1 & 0 & 0 \\
0 & 1 & 0 & 0
\end{smallmatrix}\right)\text{ and }
\rho_0(l)=
\exp\left(\begin{smallmatrix}
0 & 0 & -2\sqrt3 &2\sqrt3 \\
0 & 0 & 0 & 0\\
2\sqrt3 &0 & 0 & 0 \\
2\sqrt3 &0 & 0 & 0 
\end{smallmatrix}\right)\,.
\]
Hence $A_{0}= M
\left(\begin{smallmatrix}
1 & 0 & 0 &0 \\
0 & -1 & 0 & 0\\
0 &0 & 1 & 0 \\
0 &0 & 0 & 1 
\end{smallmatrix}\right)$
for some $M$ in the parabolic group that fixes
$v_+=\operatorname{Fix}(\langle \rho_0(m),\rho_0(l)\rangle )$, 
and that maps $v_-$, the point fixed by the parabolic group containing $\rho_0(y)$, to 
$\rho_0(yx^{-1}) \cdot v_-$, because $\varphi_0(y)= y x^{-1} y^{-1} x y^{-1}$.
With respect to our normalization we have
\[ v_+=\begin{pmatrix}
    0 \\ 0 \\ 1 \\ 1
   \end{pmatrix}
,\quad
v_-=\begin{pmatrix}
    0 \\ 0 \\ -1 \\ 1
   \end{pmatrix} \quad\text { and }\quad
\rho_0(yx^{-1}) \cdot v_- =
 \begin{pmatrix}
    \sqrt 3 \\ -1 \\ 0 \\ 2
   \end{pmatrix} . 
\] 
Hence
\begin{align*}
M &= \exp
\begin{pmatrix}
0 & 0 & -\sqrt3/2 &\sqrt3/2 \\
0 & 0 & 1/2 & -1/2\\
\sqrt3/2 & -1/2 & 0 & 0 \\
\sqrt3/2 & -1/2 & 0 & 0
\end{pmatrix} 
\intertext{ and }
A_0 &=
\begin{pmatrix}
1 & 0 & -\sqrt3/2 &\sqrt3/2 \\
0 & -1 & 1/2 & -1/2\\
\sqrt3/2 & 1/2 & 1/2 & 1/2 \\
\sqrt3/2 & 1/2 & -1/2 & 3/2
\end{pmatrix}.
\end{align*}

Let us consider the two cocycles $z_m,z_l\co\pi_1(\partial M)\to\mathfrak{v}_{\rho_0}$ which were constructed in the proof of Lemma~\ref{lem:angles}:
$z_m\co  \pi_1(\partial M)  \to  \mathfrak{v}_{\rho_0}$ given by
$z_m(l)= 0$ and  $z_m(m) = a_{l}$ where
\[
a_l = 
\begin{pmatrix}
- 1 & & & \\ 
& 3 & & \\
& & -1 & \\
& & & -1
\end{pmatrix} \in  \mathfrak v,
\]
and
$z_l\co \pi_1(\partial M)  \to  \mathfrak{v}_{\rho_0}$ given by
$z_l(l)= a_m$ and  $z_l(m) = 0$
where
\[
a_m=
\begin{pmatrix}
3 & & &  \\ 
& -1 & & \\
& & -1 & \\
& & & -1
\end{pmatrix} \in  \mathfrak v\,.\]
These cocycles satisfy:
\begin{alignat*}{2}
i_m^*([z_m]) &\neq 0, \quad & \quad i_l^*([z_m]) &= 0, \\
i_m^*([z_l]) &= 0, \quad & \quad i_l^*([z_l]) &\neq 0\,.
\end{alignat*}
Moreover we have
\begin{align*}
\varphi_0^* z_m (m) &= \Ad_{A_0^{-1}} z_m(m^{-1})\\ 
&= - \Ad_{A_0^{-1}} \Ad_{\rho_0(m)^{-1}} a_l\\
&= - \begin{pmatrix}
-1 & 0& 0& 0 \\ 
0& 3 & 2 & -2 \\
0& 2& 0 & -9\\
0& 2& 9& -2
\end{pmatrix} 
\intertext{and}
\varphi_0^* z_m (l) &= \Ad_{A_0^{-1}} z_m(l) =0\,.
\end{align*}

Since 
$$\langle i_m^*\varphi_0^* z_m, a_m\rangle =
B(a_m, \varphi_0^*z_m (m))=32=- B(a_m,a_l)$$ 
it follows that $i_m^* \varphi_0^*z_m\sim -i_m^* z_m$ (see the argument at the end of  the proof of Lemma~\ref{lem:angles}). On the other hand we have:
\[ \varphi_0^* z_l (m)=0 \text{ and } 
\varphi_0^* z_l (l) = \Ad_{A_0^{-1}}(a_m) =
\begin{pmatrix}
3 & 0& -2\sqrt 3& 2\sqrt 3 \\ 
0& -1 & 0 &0 \\
-2\sqrt 3& 0 & 2 & -3\\
-2\sqrt 3& 0 & 3 &- 4
\end{pmatrix} .
\]
Since $B(a_l, \varphi_0^* z_l (l) ) = -32 =B(a_l,a_m)$ it follows that
$i_l^*\varphi_0^* z_l \sim i_l^*z_l$. 
\end{proof}

\begin{cor}\label{cor:3}
For sufficiently large $n\in\mathbf{N}$ the composition
\[
H^1(M, \mathfrak{v}_{\rho_{1/n}}) \hookrightarrow
H^1(\partial M, \mathfrak{v}_{\rho_{1/n}})\to
H^1(m, \mathfrak{v}_{\rho_{1/n}})
\]
is the zero map.
\end{cor}
\begin{proof}
The longitude $l$ is a rigid slope (see Remark~\ref{rem:flex}). 
Thus by Lemma~\ref{lem:2} the map 
$\varphi_0^*\co H^1(M, \mathfrak{v}_{\rho_{0}})\to H^1(M, \mathfrak{v}_{\rho_{0}})$ is the identity.

Next notice that for $n$ sufficiently large, by Corollary~\ref{cor:cohmologydimU}  we have an inclusion
\[
H^1(M, \mathfrak{v}_{\rho_{1/n}}) \hookrightarrow
H^1(\partial M, \mathfrak{v}_{\rho_{1/n}})\,.
\]
 The eigenvalues of  
 $\varphi_{1/n}^*\co H^1(\partial M, \mathfrak{v}_{\rho_{1/n}}) \to
H^1(\partial M, \mathfrak{v}_{\rho_{1/n}})$  are $\pm 1$ since 
the restriction of $\varphi_{1/n}$ to the subgroup generated by $m$ and $l$ is an involution. Moreover, $\varphi_{1/n}^*$ preserves
$H^1(M, \mathfrak{v}_{\rho_{1/n}}) \hookrightarrow
H^1(\partial M, \mathfrak{v}_{\rho_{1/n}})$ and hence the induced map $\varphi_{1/n}^*$ on $H^1(M, \mathfrak{v}_{\rho_{1/n}})$ is 
$\pm\Id$ and by continuity this restriction is the identity.

On the other hand we have $\varphi_{1/n}(m) = m^{-1}$, hence by Lemma~\ref{lem:2} 
and Lemma~\ref{lem:rankone},
$\varphi_{1/n}^*$ induces  $-\Id$ on the image of $H^1(\partial M, \mathfrak{v}_{\rho_{1/n}})\to H^1(m,\mathfrak{v}_{\rho_{1/n}})$.

\end{proof}

\subsubsection{The orbifold cohomology}

It will be convenient  to consider orbifold cohomology with twisted coefficients. We follow the simplicial approach of Paragraph~\ref{subsection:twistedcohomology}.
Consider a CW-complex structure on the underlying manifold of $\mathcal O_n$ ($S^3$ in this case), that respects the stratification of the singular locus (i.e.
the singular locus is a subcomplex). Following \cite{Porti}, we  use precisely the same definition as in Paragraph~\ref{subsection:twistedcohomology}
of twisted simplicial chains and cochains to defined the 
homology and cohomology of $\mathcal O_n$ with twisted coefficients. In particular we are interested in:
$$
H^*({\mathcal O}_n, \mathfrak{v}_{\rho_{1/n}}).
$$
 The fastest way to see that these cohomology groups are independent of the CW-complex structure   is using the  the cyclic regular covering 
$  {{M}_n}\to\mathcal O_n$ that is a manifold, the Fibonacci manifold,
as mentioned at the beginning of Subsection~\ref{subsec:cohom}. We denote the projection
$$
p\co M_n\to \mathcal O_{n}.
$$
On the other hand, let $t_{n}\co M_n\to M_n$
denote the generator of the group of deck transformations, so that $\mathcal O_n=M_n/\langle t_{n}\rangle $.  It acts on cohomology, and 
$H^*( {\mathcal O}_n,\mathfrak{v}_{\rho_{1/n}})^{t_{n}^*}$ denotes the invariant subspace.
The following lemma uses the standard transfer argument and can be found in \cite{Porti}:

\begin{lemma}
 \label{lemma:transfer}
The projection induces an isomorphism
$$
p^*: H^*({\mathcal O}_n, \mathfrak{v}_{\rho_{1/n}})\overset\cong\longrightarrow H^*(  {M}_n, \mathfrak{v}_{\rho_{1/n}})^{ t_{n}^*}.
$$
\end{lemma}

It follows from this lemma that $H^*({\mathcal O}_n, \mathfrak{v}_{\rho_{1/n}})$ is independent of the CW-complex.
It is also used in the next lemma.

\begin{lemma}\label{lem:4}
There is a natural isomorphism 
$H^*({\mathcal O}_n, \mathfrak{v}_{\rho_{1/n}}) \cong
H^*(\Gamma_{1/n}, \mathfrak{v}_{\rho_{1/n}})$.
\end{lemma}

\begin{proof}
As above let $M_n \to {\mathcal O}_n$ be cyclic regular covering corresponding to the Fibonacci manifold.
The compact, hyperbolic manifold $M_n$ is aspherical,
hence there is a canonical isomorphism
\[ H^*(\pi_1(M_n), \mathfrak{v}_{\rho_{1/n}})\cong
H^*(M_n,\mathfrak{v}_{\rho_{1/n}})\,.\]
Then the lemma follows because 
$H^*({\mathcal O}_n, \mathfrak{v}_{\rho_{1/n}})\cong H^*(  {M}_n, \mathfrak{v}_{\rho_{1/n}})^{ t_{n}^*}$,
by Lemma \ref{lemma:transfer}, and 
$H^*(\pi_1({\mathcal O}_n),\mathfrak{v}_{\rho_{1/n}}) = 
H^*(\pi_1(M_n),\mathfrak{v}_{\rho_{1/n}})^{t_{n}^*} $, by the transfer in group cohomology (see \cite{Bro82} for instance).
\end{proof}

The point of working with orbifold cohomology with twisted coefficients is that one can apply  some tools of simplicial cohomology,
 mainly Mayer-Vietoris and the long exact sequence of the pair \cite{Porti}.
This will be useful in the following paragraph.

\subsubsection{A finite order automorphism of $\Gamma_{1/n}$}

\begin{prop}\label{prop:5}
For sufficiently large $n\in\NN$ we have
\begin{enumerate}
\item $H^1(\Gamma_{1/n},\mathfrak{sl}(4)_{\rho_{1/n}})\cong 
H^1(\Gamma_{1/n},\mathfrak{v}_{\rho_{1/n}})\cong\RR$ is one-dimensional and $\varphi^*_{1/n}$ acts trivially on it.
\item $H^2(\Gamma_{1/n},\mathfrak{sl}(4)_{\rho_{1/n}})\cong 
H^2(\Gamma_{1/n},\mathfrak{v}_{\rho_{1/n}})\cong\RR$ is one-dimensional and $\varphi^*_{1/n}$ acts by multiplication by $-1$ on it.
\end{enumerate}
\end{prop}

\begin{proof}
We start with the decomposition 
\[
H^*(\Gamma_{1/n},\mathfrak{sl}(4)_{\rho_{1/n}})=
H^*(\Gamma_{1/n},\mathfrak{so}(3,1)_{\rho_{1/n}})\oplus
H^*(\Gamma_{1/n},\mathfrak{v}_{\rho_{1/n}}).
\]
The group $H^1(\Gamma_{1/n},\mathfrak{so}(3,1)_{\rho_{1/n}})=0$ vanishes by Weil's infinitesimal rigidity and hence 
$$H^2(\Gamma_{1/n},\mathfrak{so}(3,1)_{\rho_{1/n}})=0$$ by Poincar\'e duality and Lemma~\ref{lem:4}. Thus
\[ H^i(\Gamma_{1/n},\mathfrak{sl}(4)_{\rho_{1/n}})=H^i(\Gamma_{1/n},\mathfrak{v}_{\rho_{1/n}}) \text{ for $i=1,2$.}\]
In order to compute 
$H^i(\Gamma_{1/n},\mathfrak{v}_{\rho_{1/n}})\cong
H^i({\mathcal O}_n,\mathfrak{v}_{\rho_{1/n}})$ we shall apply the Mayer-Vietoris sequence to the decomposition ${\mathcal O}_n=M\cup N_n$ where 
$N_n=\mathcal N(\mathrm{Sing}({\mathcal O}_n))$ is a regular neighborhood of the singular locus such that $M\cap N_n = \partial M$.
By Corollary~\ref{cor:cohmologydimU}, the dimension of $H^i(M,\mathfrak{v}_{\rho_{1/n}})$ is $1$ for $i=1,2$ and zero otherwise.
Hence
\begin{align*}
H^0({\mathcal O}_n,\mathfrak{v}_{\rho_{1/n}}) &\cong 
H^0(M,\mathfrak{v}_{\rho_{1/n}})\cong \mathfrak{v}^{\rho_{1/n}(\pi_1(M))} =0\\
\intertext{and}
H^0(\partial M,\mathfrak{v}_{\rho_{1/n}}) &\cong \mathfrak{v}^{\rho_{1/n}(\pi_1(\partial M))} = \mathfrak{v}^{\rho_{1/n}(\pi_1(N_n))} 
\cong
H^0(N_n,\mathfrak{v}_{\rho_{1/n}}).
\end{align*}
Therefore we obtain the following exact sequence from Mayer-Vietoris
\[
 H^1({\mathcal O}_n,\mathfrak{v}_{\rho_{1/n}}) \rightarrowtail
H^1(M,\mathfrak{v}_{\rho_{1/n}})\oplus H^1(N_n,\mathfrak{v}_{\rho_{1/n}})\to
H^1(\partial M,\mathfrak{v}_{\rho_{1/n}})\twoheadrightarrow
H^2({\mathcal O}_n,\mathfrak{v}_{\rho_{1/n}})\,.
\]
The injectivity of the first arrow follows from a dimension counting of the 0th-cohomology terms.
One can deduce that the last arrow is surjective by looking at the terms that follow the exact sequence:
 $\dim H^2( M,\mathfrak{v}_{\rho_{1/n}})=1$ by Corollary~\ref{cor:cohmologydimU}, 
$\dim H^2( \partial M,\mathfrak{v}_{\rho_{1/n}})=1$ by Lemma~\ref{lem:rankone}(i),
and the other cohomology groups appearing vanish as $N_n$ has the homotopy type of a circle
and $M$ and $\partial M$ have the homotopy type of a 2-complex. 
By Corollary~\ref{cor:3}, both groups 
$H^1(M,\mathfrak{v}_{\rho_{1/n}})$ and 
$H^1(N_n,\mathfrak{v}_{\rho_{1/n}})$ have the same image in
$H^1(\partial M,\mathfrak{v}_{\rho_{1/n}})$ which is exactly the kernel of the map 
$H^1(\partial M,\mathfrak{v}_{\rho_{1/n}})\to 
H^1(m,\mathfrak{v}_{\rho_{1/n}})$.
Notice  also that 
$\dim H^1(\partial M,\mathfrak{v}_{\rho_{1/n}}) =2$
and 
\[\dim H^1(N_n,\mathfrak{v}_{\rho_{1/n}}) =
\dim H^0(N_n,\mathfrak{v}_{\rho_{1/n}}) = 
\dim \mathfrak{v}^{\rho_{1/n}(\pi_1 N_n) } =1\,.\]
Therefore we get $\dim  H^1({\mathcal O}_n,\mathfrak{v}_{\rho_{1/n}}) =1$.
Moreover, the map $\varphi_{1/n}^*$ acts trivially on 
$H^1( \mathcal O_n,\mathfrak{v}_{\rho_{1/n}})$ since by the proof of
Corollary~\ref{cor:3} it acts trivially on 
$H^1( M,\mathfrak{v}_{\rho_{1/n}})$, and 
$H^1({\mathcal O}_n,\mathfrak{v}_{\rho_{1/n}})$ injects into 
$H^1( M,\mathfrak{v}_{\rho_{1/n}})$.

On the other hand we have
\[
H^1(\partial M,\mathfrak{v}_{\rho_{1/n}})\cong
H^1(\partial M,\RR)\otimes \mathfrak{v}^{\rho_{1/n}(\pi_1\partial M)},
\]
$\varphi(m)=m^{-1}$ and $\varphi(l)=l$. Hence the eigenvalues of 
$\varphi_{1/n}^*\co H^1(\partial M,\mathfrak{v}_{\rho_{1/n}})\to
H^1(\partial M,\mathfrak{v}_{\rho_{1/n}})$ are $\pm1$.
The eigenspace corresponding to the eigenvalue $+1$ is
the image of $H^1( M,\mathfrak{v}_{\rho_{1/n}})$ (and $H^1(N_n,\mathfrak{v}_{\rho_{1/n}})$). 
Hence $\varphi_{1/n}^*$ acts as $-\Id$ on 
$H^2({\mathcal O}_n,\mathfrak{v}_{\rho_{1/n}})$.
\end{proof}

\subsubsection{Deforming the projective structure of $\mathcal O_ n$}
\label{subsubsection:deform}

\begin{proof}[Proof of Proposition~\ref{prop:nonrigid}]
We shall show that every Zariski tangent vector 
is integrable.
We  use the following general setup:
let $\Gamma$ be a  finitely presented group and let
$\rho\co \Gamma \to GL(r,\mathbf R)$ be a representation.
A \emph{formal deformation} of $\rho$ is a representation
$\rho_t\co \Gamma\to GL(r,  \mathbf R [[t]]) $ such that
$ \rho_0  = \rho$.
Here $ \mathbf R [[t]] $ denotes the ring of formal power
series and $\rho_0\co\Gamma\to\CC $ is the evaluation of
$\rho_t$ at $t=0 $. 

Every formal deformation $\rho_t$ of $\rho$ can be written in
the form
$$  \rho_t ( \gamma) =(I_r + t u_1(\gamma) + t^2 u_2(\gamma) + 
        \cdots ) \rho(\gamma) $$
where $I_r$ denotes the identity matrix and 
$u_i\co \Gamma\to \mathfrak{gl}(r)$ are maps i.e.\ elements of 
$C^1( \Gamma,\mathfrak{gl}(r)_\rho)$.
An easy calculation gives
that $u_1 \in Z^1( \Gamma,\mathfrak{gl}(r)_\rho)$ is a cocycle  (Weil's theorem).
More generally we have the following:

\begin{lemma} \label{lem:def}
        Let $\rho\co\Gamma\to GL(r)$ be a homomorphism.
        Then $\rho_{t}\co \Gamma\to GL(r,\mathbf R [[t]])$ given by
        $$ \varrho_{t}(\gamma) = 
         (I_r+t u_1(\gamma)+t^2 u_2(\gamma)+t^3 u_3(\gamma)+\cdots)
         \rho(\gamma)$$
         is a homomorphism if and only if for all $k\in\ZZ$, $k\geq 1$, we 
         have 
         \begin{equation}
	    \label{eqn:integrability}
          \delta u_{k} + \sum_{i=1}^{k-1} u_i \cupd u_{k-i} = 0\,.
         \end{equation}
 \end{lemma}

The proof of this lemma is a direct calculation.
Here the cup product $\cupd$ is  the composition of the usual cup product $\cup$ with the matrix multiplication
$$
C^1(\Gamma,\mathfrak{gl}(r)_\rho)\otimes 
C^1(\Gamma,\mathfrak{gl}(r)_\rho) \xrightarrow{\cup}
C^2(\Gamma,\mathfrak{gl}(r)_\rho \otimes \mathfrak{gl}(r)_\rho)\xrightarrow{\boldsymbol{\cdot}}
C^2(\Gamma,\mathfrak{gl}(r)_\rho).
$$
Namely given to cochains $c_1,c_2\in C^1( \Gamma,\mathfrak{gl}(r)_\rho)$ the cup product $c_1\cupd c_2 \in C^2( \Gamma,\mathfrak{gl}(r)_\rho)$ is given by 
\[ c_1\cupd c_2 (\gamma_1, \gamma_2)=
c_1(\gamma_1) \Ad_{\rho(\gamma_1)} \big(c_2(\gamma_2)\big)\in \mathfrak{gl}(r)\,.\]

In the sequel the representation $\rho$ is going to be always 
$\rho_{1/n}$, hence we omit it from notation. Note that the $\Gamma_{1/n}$--module $\mathfrak{gl}(4)$ decomposes as a direct sum
\[ \mathfrak{gl}(4) = \mathbf R  \oplus \mathfrak{sl}(4)\]
where $\mathbf R \cong \mathbf R \cdot I_4$ is the trivial module, it is the center of $\mathfrak{gl}(4)$.
Moreover $H^i(\Gamma_{1/n},\mathbf R)=0$ for $i=1,2$ since
$H_1(M_n,\mathbf Z)$ is finite (no root of unity is a zero of the Alexander polynomial of the figure eight-knot). Hence
$$H^i(\Gamma_{1/n},\mathfrak{gl}(4))=
H^i(\Gamma_{1/n},\mathfrak{v})\text{ for $i=1,2$.}$$

Instead of $\varphi_{1/n}$ we shall consider the automorphism $\psi_n\co \Gamma_{1/n}\to\Gamma_{1/n} $ given by 
$\psi_{1/n}=c_{y^{-1}}\circ \varphi_{1/n}$, where $c_{y^{-1}}$ denotes conjugation by $y^{-1}$. By Proposition~III.8.3 from
\cite{Bro82}, the induced maps in cohomology are the same: $\psi_{1/n}^*=\varphi_{1/n}^*$. Notice that $\psi_{1/n}^4$ is the identity. 

Let $v\in H^1(\Gamma_{1/n},\mathfrak{gl}(4))$, we choose a cocycle $u_1\in  Z^1(\Gamma_{1/n},\mathfrak{gl}(4))$ in its cohomology class.
Since $\psi_{1/n}$ has order $4$, we may consider the average of the action of $\psi_{1/n}^*$ on $u_1$:
$$
\frac14 (u_1+ \psi_{1/n}^* (u_1)+ (\psi_{1/n}^*)^2 (u_1) + (\psi_{1/n}^*)^3 (u_1) ).
$$
This cocycle is $\psi_{1/n}^*$--invariant and, since $\psi_{1/n}^*$ acts as the identity on  $H^1(\Gamma_{1/n},\mathfrak{gl}(4))$,
it is cohomologous to $u_1$. Thus  we may assume that 
$\psi_{1/n}^* (u_1)=u_1$ by averaging.

First we claim that $u_1\cupd u_1$ is cohomologous to zero. This is because $\psi_{1/n}^* (u_1\cupd u_1)= \psi_{1/n}^* (u_1)\cupd \psi_{1/n}^* (u_1)=u_1\cupd u_1$
and $\psi_{1/n}^* (u_1\cupd u_1)$ is cohomologous to $-u_1\cupd u_1$, as $u_1\cupd u_1$ is a 2-cocycle and $\psi_{1/n}^*$ acts on 
 $H^2(\Gamma_{1/n},\mathfrak{gl}(4))$ by multiplication by $-1$.

There exist a $1$--chain $u_2 \in C^1(\Gamma_{1/n},\mathfrak{gl}(4))$ satisfying $u_1\cupd u_1 + \delta u_2 =0$. As before, we may assume that 
$\psi_{1/n}^* (u_2)=u_2$ by averaging. The same argument as before proves that  
$$
u_2\cupd u_1+u_1\cupd u_2=\psi_{1/n}^* (u_2\cupd u_1+u_1\cupd u_2) \sim- (u_2\cupd u_1+u_1\cupd u_2),
$$ 
where $\sim$ means cohomologous. Thus there exists a 
$1$--chain  $u_3 \in C^1(\Gamma_{1/n},\mathfrak{gl}(4))$ satisfying 
$$
u_1\cupd u_2+u_2\cupd u_1 + \delta u_3 =0.
$$
Again $u_3$ can be chosen to be $\psi_{1/n}^*$--invariant. By induction we can find an infinite sequence
of $1$--chains $u_2,u_3,\ldots$ that satisfy Equation~(\ref{eqn:integrability}).

This implies  that all obstructions to integrability vanish, hence we have a formal deformation 
of $\rho$, that gives formal integrability of $v$. We apply Artin's theorem~\cite{Art68}, to conclude that formal
 integrability implies actual integrability of $v$.
\end{proof}

\subsection{The Whitehead link}

A similar computation as for the figure eight knot tells us that the Whitehead link $L=K_1\sqcup K_2$ is infinitesimally projectively rigid. Let $\Gamma=\pi_1(M)$
denote the fundamental group of the Whitehead link exterior $M$.
We will work with the presentation:
$$
\Gamma=\langle x, y\mid x y^{-1} x^{-1} y x^{-1} y^{-1} x y x^{-1} y x y^{-1} x y x^{-1} y^{-1} \rangle
$$
where $x$ is a meridian for $K_1$ and $y$ is a meridian for $K_2$.
The holonomy representation $\rho\co\Gamma\to SL_2(\mathbf C)$ is given by 
$$
x\mapsto   \left( \begin {array}{cc} 1&1\\\noalign{\medskip}0&1\end {array}
\right)
\qquad
y\mapsto  \left(       \begin {array}{cc} 1&0\\\noalign{\medskip} -1-\mathbf{i} &
1\end {array}         \right)
$$
(see \cite{RileyLink} for details). A computation analogous to the one of the previous subsection  
shows that $\operatorname{dim} H^1(M; \mathfrak{v}_{\Ad\rho})=2$. Hence, by Corollary~\ref{cor:dim}, the Whitehead link 
is infinitesimally projectively rigid relative to the cusps.

Once we know the dimension of the deformation space, we have a geometric tool to understand the deformations: let $S$ denote the thrice puncture sphere   
illustrated in Figure~\ref{fig:whitehead}. By symmetry of the components of the link, there are two of them. The surface $S$ intersects one boundary torus in a longitude 
$l_x = y x^{-1}y^{-1}x y^{-1} x^{-1} y x $, and the other one
in two meridians $y$ and $z=x^{-1} y^{-1}x y x^{-1} y x$, with opposite orientation.
The restriction of the holonomy onto $\pi_1 (S)$ is conjugate to a representation into $SL_2(\mathbf R)$. Hence $S$ a totally geodesic thrice puncture sphere in the link complement.

\begin{figure}
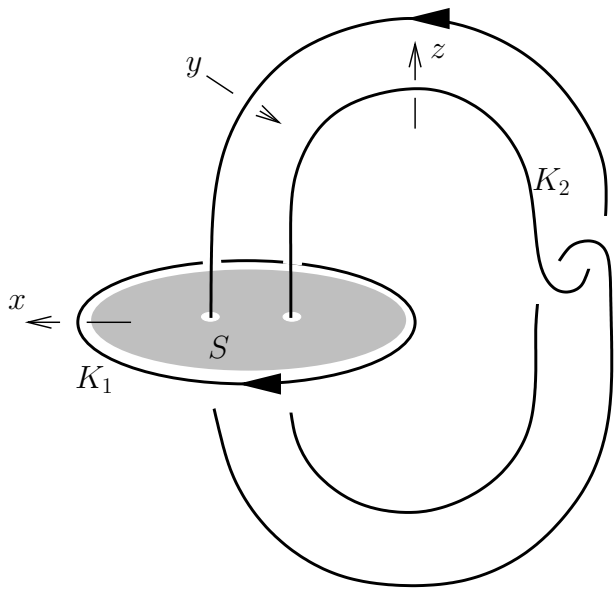

\begin{center}
\resizebox{8cm}{!}{\input whitehead2.pstex_t}
\end{center}
   \caption{The thrice punctured sphere $S$ in the Whitehead link.}\label{fig:whitehead}
\end{figure}

\begin{lemma}
\label{lem:longitudeWhitehead}
Let $\partial_1 M$ denote the boundary component of $K_1$. 
Every slope on $\partial_1 M$ different from the longitude $l_x$ is a rigid slope.
\end{lemma}

\begin{proof}
We consider the bending along $S$ (see \cite{JohnsonMillson} for more details about the bending construction). Since $S$ is totally geodesic, the image of its holonomy is contained in 
$ SO(2,1)\subset SO(3,1)$. On the other hand the $SO(2,1)$ commutes with the exponential of
$$
a=
\begin{pmatrix}
   -3 & & & \\
    & 1 & & \\
    & & 1 &  \\
    & & & 1
\end{pmatrix}
\in \mathfrak{sl}(4).
$$
We write $\Gamma$ as an HNN-extension
$
\Gamma=\pi(M\setminus S)*_{\pi_1(S)}.
$
In particular we have a generator $\nu$ of $\Gamma$ such that the only relations involving $\nu$ are 
of the form $\nu j_1(s)\nu^{-1}=j_2(s)$, $\forall s\in\pi_1(S)$, where $j_1,j_2: \pi_1(S)\to \pi_1(M\setminus S)$ are the morphisms
induced by inclusions of each copy of $S$ in $M\setminus \mathcal N( S)$. The bending is the family of representations $\rho_t$, $t\in\mathbf R$,
such that $\rho_t\vert_{\pi_1(M\setminus S)}=\rho$ and  $\rho_t(\nu)=\exp(t\, a)\rho(\nu)$.
Johnson and Millson  prove in \cite[Lemma~5.5]{JohnsonMillson} that the cocycle tangent to this deformation is not cohomologous to zero.

If we restrict this bending cocycle to  
$\partial_1 M$, it is itself a bending cocycle along the longitude
$l_x$, and it happens to be precisely the infinitesimal deformation constructed in the proof of Lemma~\ref{lem:angles}. 
Thus, except for the longitude itself, this infinitesimal deformation is nontrivial when restricted to any slope of the torus,
because the cusp shape of the Whitehead link lies in the Gaussian integers $\mathbf Z[\mathbf i ]$,
thus the angle of any slope with the longitude $l_x$ can never be $\pi/3$, and we can apply Lemma~\ref{lem:angles}.
\end{proof}

\begin{proof}[Proof of Proposition~\ref{prop:twist}]
Lemma~\ref{lem:longitudeWhitehead} and Proposition~\ref{prop:flexingslope} imply that for almost all $n$ the $(n,1)$--Dehn fillings are infinitesimally projectively rigid.
According to \cite{Akiyoshi} those fillings are precisely the punctured torus bundles with tunnel number one. 

Twists knots are obtained by $(1,n)$--Dehn fillings, but we cannot apply Proposition~\ref{prop:flexingslope}, because the longitude is not a rigid slope.
However, the path  $(p,q)=(1,s)$ for $s\in\mathbf R$ and $s\geq 1$ is contained in the whole deformation space (cf.~\cite{Akiyoshi}). Hence, since the coefficients $(1,1)$ correspond to the figure eight knot exterior, 
with an argument similar to 
Theorem~\ref{thm:fromonetomany}, the $(1,n)$--Dehn fillings are infinitesimally rigid for all but finitely many $n$.
\end{proof}

%
%
%
\bibliographystyle{plain}
\bibliography{Heusener-Porti}

\end{document}